\documentclass[12pt]{amsart}
\usepackage{graphicx}
\usepackage{tikz-cd}
\usepackage{wrapfig}
\usepackage[margin=1.25in]{geometry}
\usepackage[mathscr]{euscript}
\usepackage{dsserif}
\usepackage{setspace}
\usepackage{amsmath}
\usepackage{amssymb} 
\usepackage{amsthm}
\usepackage{mathtools}
\usepackage{enumitem}
\usepackage{thmtools}
\usepackage{xcolor}
\usepackage{hyperref}
\usepackage[nameinlink]{cleveref}
\hypersetup{
	colorlinks = true,
	linkcolor= [RGB]{105, 2, 184},
	citecolor= [RGB]{105, 2, 184},
}

\def\R{\mathbb{R}}
\def\Z{\mathbb{Z}}
\def\H{\mathbb{H}}

\DeclareMathOperator{\Teich}{\mathcal{T}}
\DeclareMathOperator{\Isom}{Isom}

\DeclareMathOperator{\Ends}{Ends}
\DeclareMathOperator{\Map}{MCG}
\DeclareMathOperator{\Homeo}{Homeo}
\DeclareMathOperator{\diam}{diam}

\DeclareMathOperator{\AH}{\textrm{AH}}

\DeclareMathOperator{\id}{id}

\DeclareMathOperator{\capgen}{\mathbb{g}}
\DeclareMathOperator{\dS}{\partial\mathfrak{E}}
\DeclareMathOperator{\SI}{\mathcal{S}\mathcal{I}}
\renewcommand{\bar}[1]{\overline{#1}}
\renewcommand{\int}{\textrm{int}}

\newcommand{\Lam}{\Lambda}
\newcommand{\gammaSig}{\gamma_\Sigma}

\newcommand{\phiSig}{\phi_{\Sigma}}
\newcommand{\M}{\bar M}
\newcommand{\Mf}{\bar{M_f}}
\newcommand{\Cmin}{C_{\textrm{min}}}
\newcommand{\jpp}{\ensuremath{J_+^{\geq k}}}
\newcommand{\jmp}{\ensuremath{J_-^{\geq k}}}
\newcommand{\jmm}{\ensuremath{J_-^{\leq k}}}
\newcommand{\jpm}{\ensuremath{J_+^{\leq k}}}
\newcommand{\eL}{\mathcal{L}}
\newcommand{\AC}{\mathcal{AC}}
\newcommand{\mathcalO}{\mathcal{O}}
\newcommand{\pcalO}{p_{\mathcalO}}
\renewcommand{\tilde}{\widetilde}
\usepackage{stmaryrd}

\theoremstyle{plain}
\newtheorem{theorem}{Theorem}[section]
\newtheorem{proposition}[theorem]{Proposition}
\newtheorem*{theorem*}{Theorem}
\newtheorem{mytheorem}{Theorem}

\newtheorem{lemma}[theorem]{Lemma}
\newtheorem{corollary}[theorem]{Corollary}
\newtheorem{question}[theorem]{Question}

\newtheorem{example}[theorem]{Example}
\newtheorem{remark}[theorem]{Remark}
\newtheorem{conjecture}[theorem]{Conjecture}
\newtheorem{claim}[theorem]{Claim}
\newtheorem{observation}[theorem]{Observation}

\newtheorem{casee}{Type}

\newenvironment{enumeri}{\begin{enumerate}[label=(\roman*)]}{\end{enumerate}}

\begin{document}
		\begin{abstract}
			Let $S$ be a boundaryless infinite-type surface with finitely many ends and consider an end-periodic homeomorphism $f$ of S. The end-periodicity of $f$ ensures that its associated mapping torus $M_f$ has a compactification as a $3$-manifold with boundary; further, if $f$ is atoroidal then $M_f$ admits a hyperbolic metric. Such maps admit invariant \emph{positive and negative Handel--Miller laminations}, $\Lambda^+$, $\Lambda^-$, whose leaves naturally project to the arc and curve complex of a compact subsurface $Y\subset S$.
		 
			As an end-periodic analogy to work of Minsky in the finite-type setting, we show that there exists a hyperbolic structure $s$ of $M_f$ such that for every $\varepsilon>0$ there exists $K> 0$ (depending only on $\varepsilon$ and the \emph{capacity} of $f$) for which $d_Y (\Lambda^+, \Lambda^-)\geq K$ implies $\ell_s(\partial Y) \leq \varepsilon$. Here $\ell_s (\partial Y)$ denotes the total length of the geodesic representative of $\partial Y$ in $(M_f, s)$.
		
			This work additionally produces the following: for every $\varepsilon>0$ and closed connected surface $\Sigma$ of genus $g\geq 2$, we provide a closed fibered hyperbolic $3$-manifold in which $\Sigma$ embeds as a totally geodesic surface whose systole length is at most $\varepsilon$.
		\end{abstract}
	
		\title{Short curves of end-periodic mapping tori}  
		\author{Brandis Whitfield}
	\maketitle
	\setcounter{tocdepth}{2}
	\tableofcontents
	
	\section{Introduction}
		The interplay between laminations, mapping classes of surfaces and hyperbolic $3$-manifolds is central to the study of surfaces and the geometry of their mapping tori. 
		The geometry of a closed fibered hyperbolic $3$-manifold may be further understood by analyzing the combinatorial data of its monodromy.
		 While the connections between such objects are well-studied in the finite-type setting, our work examines fibered hyperbolic $3$-manifolds constructed via end-periodic homeomorphisms of infinite-type surfaces.
		 
		Let $S$ be a boundaryless infinite-type surface with finitely many ends, all non-planar. We say a homeomorphism $f$ of $S$ is \emph{end-periodic} if each of its ends is either \emph{attracting} or \emph{repelling} under $f$, and that it is \emph{atoroidal} if it admits no periodic curve up to isotopy. Field--Kim--Leininger--Loving \cite[Section 6]{FKLL23} provide plentiful examples of such maps as the composition of pairwise commuting handle shifts with sufficiently large powers of a pseudo-Anosov map on a compact subsurface of $S$. In recent work, Field--Kent--Leininger--Loving \cite[Section 2]{FKeLL23} introduce a description of the complexity of an end-periodic map by the pair $(\chi(f),\xi(f))$ known as the \textit{capacity} of f. Intuitively, $\xi(f)$ measures the amount of shifting in and out of the ends and $\chi(f)$ is the (absolute value of the) Euler characteristic of a compact subsurface known as a \textit{minimal core} of the map $f$.
		We refer the reader to \Cref{subsec:cores} for precise definitions of these terms.\\
		
		The \textit{mapping torus} of $f$ is the boundaryless fibered $3$-manifold with monodromy $f$ defined by the quotient $$M_f := S \times [0,1] / (x,1)\sim(f(x),0).$$  Interestingly, the end-periodic nature of the homeomorphism allows its mapping torus to admit a compactification $\Mf$ as a $3$-manifold with boundary (see e.g. \Cref{mfconstruct}). The structure of the homeomorphism's action on curves determine topological properties of the manifold. In particular, if $f$ is atoroidal, then $M_f$ is atoroidal and admits a hyperbolic structure by Thurston's Hyperbolization Theorem for manifolds with boundary (see \cite[Proposition 3.4]{FKLL23} as stated for irreducible end-periodic maps). In general, when an irreducible manifold with incompressible boundary $M$ satisfies $\chi(\partial M) <0$,  if its interior admits a hyperbolic structure, then it admits a plethora of hyperbolic structures parameterized by the deformation space $\text{AH}(M)$ (see \cite[Theorem 8.4]{Kap01}). 
		
		In unpublished work carried out during the 1990s, Handel and Miller developed a lamination theory for end-periodic maps introducing the \emph{positive} and \emph{negative Handel--Miller laminations} $\Lambda^+$ and $\Lambda^-$ of $S$. There are many parallels between the Handel--Miller laminations of an end-periodic homeomorphism and the unstable and stable measured laminations $\lambda^+, \lambda^-$ of a pseudo-Anosov homeomorphism of a closed surface (see \cite[Section 8]{LMT23}, \cite[Section 4]{CCF21}).  From the perspective of a compact subsurface $Y \subset S$, the leaves of the positive (resp. negative) Handel--Miller lamination project to $\pi_Y(\Lambda^\pm)$, a finite (possibly empty), disjoint collection of arcs in the arc and curve complex $\mathcal{AC}(Y)$ of $Y$. From there we may consider the quantity $d_Y(\Lambda^+, \Lambda^-)$ which measures the distance between the respective subsurface projections $\pi_Y(\Lambda^+)$ and $\pi_Y(\Lambda^-)$ in $\mathcal{AC}(Y)$.
		
		Given a compact essential subsurface $Y\subset S \subset M_f$ and a hyperbolic structure $s\in\AH(M_f)$, we consider $\ell_s(\partial Y)$, the total length of the geodesic representative of the multicurve $\partial Y$ in $(M_f, s)$.
		  
		\begin{mytheorem}\label{thma}
			For any $D, \varepsilon > 0$, there exists $K=K(D,\varepsilon)$ such that for any atoroidal end-periodic homeomorphism $f:S \to S$ with $\chi(f), \xi(f)\leq D$, there exists $s\in \AH(M_f)$ satisfying the following. 
			For any compact essential subsurface $Y\subset S\subset M_f$, we have:
			$$d_Y(\Lambda^+, \Lambda^-)\geq K\, \implies\, \ell_{s}(\partial Y)\leq \varepsilon.$$
		\end{mytheorem}
		
		We note that the proofs of our main theorems begin with an arbitrary atoroidal end-periodic homeomorphism $f$ and produce a hyperbolic metric $s\in \AH(M_f)$ satisfying the implication of \Cref{thma}. In particular, although the constant $K$ depends on $\varepsilon$ and $D$, the metric $s$ can actually be chosen independently of $D$ and $\varepsilon$. We hope such a statement, with constants controlled only by the capacity of the homeomorphism, holds for any hyperbolic metric in $\AH(M_f)$.
		
		When $\Mf$ is additionally known to be acylindrical, equivalently when $f$ is \textit{strongly irreducible} (see \Cref{fkllhyp}), there exists a unique metric $s^* \in \AH(M_f)$ whose convex core has totally geodesic boundary \cite[Section 4]{Mc90}. For a strongly irreducible homeomorphism $f$, let $s^*$ denote this unique structure. We wish to strengthen our main theorem so that $s^*$ realizes the bounds of \Cref{thma} in the following way:
		
		\begin{conjecture}\label{conj:tgmetric}
			For any $D, \varepsilon > 0$, there exists $K=K(D,\varepsilon)$ such that for any strongly irreducible end-periodic homeomorphism $f:S \to S$ with $\chi(f), \xi(f)\leq D$ and any compact essential subsurface $Y\subset S\subset M_f$, we have:	
			$$d_Y(\Lambda^+, \Lambda^-)\geq K \,\implies\, \ell_{s^*}(\partial Y)\leq \varepsilon.$$
		\end{conjecture}
		
		Such a result would gain control over the totally geodesic surfaces which arise as boundaries of compact irreducible acylindrical hyperbolic manifolds with boundary. The next theorem confirms the conjecture when $\partial M$ is restricted to consist solely of genus-$2$ surfaces. 
		
		\begin{mytheorem}\label{thmb}
			For any $D, \varepsilon > 0$, there exists $K=K(D,\varepsilon)$ satisfying the following: suppose that $f: S\to S$ is a strongly irreducible end-periodic homeomorphism such that $\partial \Mf$ consists solely of genus-$2$ surfaces and $\chi(f), \xi(f)\leq D$; then for any compact essential subsurface $Y\subset S$ $\subset M_f$, we have:
			$$ d_Y(\Lambda^+, \Lambda^-)\geq K \,\implies\, \ell_{s^*}(\partial Y)\leq \varepsilon.$$
		\end{mytheorem}
		
		Results in \Cref{examples} confirm that this phenomenon is not restricted to the genus-$2$ setting as one may pass to covers $M_{f^n}\to M_{f}$.
		
		\subsection{Totally geodesic boundary structures}
		Let $\Sigma$ be a closed surface and let $\mathcal{M}(\Sigma)$ denote its moduli space, the space of unmarked hyperbolic structures on $\Sigma$. Let $\SI(\Sigma)$ denote the collection
		of strongly irreducible homeomorphisms $f$ with $\Sigma \subset \partial \Mf$, considered over all boundaryless infinite-type surfaces with finitely many ends, all non-planar. 
		
		When $f$ is strongly irreducible, we let $(\Mf, s^*_f)$ denote its acylindrical compactified mapping torus adorned with its unique totally geodesic boundary structure. Let $\dS(\Sigma) \subset \mathcal{M}(\Sigma)$ denote the collection of unmarked totally geodesic boundary metrics recovered over all $f\in \SI(\Sigma)$, i.e.:  $$\dS(\Sigma):= \{(\Sigma, s^*_f)\,|\, f\in \SI(\Sigma)\}.$$ 
		
		\begin{question}\label{bdrystructures}
			Suppose $\SI(\Sigma)$ is non-empty. Is $\dS(\Sigma)$ dense in $\mathcal{M}(\Sigma)$?
		\end{question}
		\Cref{bdrystructures} is inspired by work of Fujii--Soma \cite{FS97} who show that $\mathcal{R}(\Sigma)$, the space of totally geodesic boundary structures represented by acylindrical manifolds with boundary homeomorphic to $\Sigma$, is dense in $\mathcal{M}(\Sigma)$. We also take inspiration from work of Kent \cite{K05} who conjectures that such a statement is true for the collection of knot complements with totally geodesic boundary. 
		\begin{question}[Kent]\label{bdryepsilonthin}
			Suppose $\SI(\Sigma)$ is non-empty. Are the systoles of $\dS(\Sigma)$ arbitrarily small?
		\end{question}
		\Cref{pinch} gives a positive answer to \Cref{bdryepsilonthin}. Our results mirror her main result in \cite{K05} which produces arbitrarily short curves of totally geodesic boundaries of knot complements. Namely, we show for every $\varepsilon>0$, $g\geq 2$, $\dS(\Sigma_g)$ escapes each $\varepsilon$-thick part of $\mathcal{M}(\Sigma)$. We note that this construction can be extended so that $\partial \Mf$ is the union of an even number of connected genus-$g$ surfaces $\Sigma^1 \sqcup \ldots \sqcup \Sigma^{2n}$. 
		
		Finally, our work in the world of infinite-type surfaces culminates in a result about closed totally geodesic surfaces in closed fibered hyperbolic $3$-manifolds.
		
		\begin{mytheorem}\label{thmc}
			For each $g\geq 2$, $\varepsilon>0$, there exists a closed fibered hyperbolic manifold $N=N(g,\varepsilon)$ satisfying the following. $N$ admits a totally geodesic embedding of a genus-$g$ hyperbolic surface $\Sigma_g$ whose systole length is at most $\varepsilon$.
			Further, $\Sigma_g$ can be made (almost) transverse to a circular pseudo-Anosov flow on $N$. 
		\end{mytheorem}
		
		 A positive answer to \Cref{bdrystructures} would show that up to arbitrarily small deformations, {every} hyperbolic surface can be realized simultaneously within a closed hyperbolic $3$-manifold as an embedded totally geodesic surface and depth-zero leaf of a depth-one foliation. We wonder if tools such as the doubling construction of Landry--Minsky--Taylor \cite{LMT23} can be carefully adapted to yield an isometric embedding of any $(\Mf, s^*)$ into a {fibered} hyperbolic manifold so that the boundary is (almost) transverse to the circular pseudo-Anosov flow. This would allow us to conclude that any hyperbolic surface, up to arbitrarily small deformation, appears as a totally geodesic surface of some fibered hyperbolic manifold. 
		
		\subsection{Motivation from the finite-type setting}
		
		Theorems A and B are directly inspired by Theorem B of Minsky \cite{M00} which formed one implication of his \textit{Bounded Geometry Theorem}, a part of the program developed to resolve the \textit{Ending Lamination Conjecture} (now the \textit{Ending Lamination Theorem} \cite{BCM12}). The Bounded Geometry Theorem establishes a combinatorial condition using the manifold's ending laminations to detect when a Kleinian surface group has bounded injectivity radius, a setting in which the Ending Lamination Conjecture was first resolved \cite{M94}.
		
		As reflected in the statement of \Cref{thma}, Minsky's theorem relates the subsurface projection distance $d_Y(\lambda^+,\lambda^-)$ between the unstable and stable laminations of a pseudo-Anosov homeomorphism to the length of the geodesic representative of $\partial Y$ in the hyperbolic mapping torus of the pseudo-Anosov homeomorphism. In this setting $\lambda^+$ and $\lambda^-$ are precisely the ending laminations of the cover corresponding to the given fiber of the hyperbolic manifold. We rephrase the theorem in this special setting \footnote{This follows from Theorem B of \cite{M00} by considering $(F\times\R, \tilde s)$, the infinite-cyclic cover of $(N_\phi, s)$ corresponding to the fiber, with its hyperbolic metric $\tilde s$ given by the pullback of $s$ under the covering map. In this setting, the ending laminations of $(F\times\R, \tilde s)$ are precisely the unstable and stable laminations $\lambda^+$, $\lambda^-$. Minsky observes that for the Bers constant $L_0$, the collection of curves $\mathcal{C}_0(\tilde s, L_0) \subset \mathcal{N}_1(\pi_Y(\lambda^+) \cup \pi_Y(\lambda^-))$. In other words, $d_Y(\lambda^+, \lambda^-) \asymp \diam_{\mathcal{C}(Y)}(\mathcal{C}_0(\tilde s, L_0))$, up to a negligible constant.}:
		
		\begin{theorem}[Minsky \label{minsky}\cite{M00}]
			Given $F$, a closed connected genus-$g$ surface, and $\varepsilon > 0$, there exists $K=K(g,\varepsilon)>0$ so that if $N$ is a hyperbolic manifold fibered by $F$ with pseudo-Anosov monodromy $\phi$ and corresponding unstable and stable laminations $\lambda^+, \lambda^-$, then: 
			$$d_Y (\lambda^+,\lambda^-) \geq K \, \implies\, \ell_{N} (\partial Y) \leq \varepsilon,$$
		where $Y$ is any essential subsurface of $F$.
		\end{theorem}
		
		We use this theorem both as motivation and a tool to extend its result to the setting of infinite-type surfaces and their atoroidal end-periodic homeomorphisms. The constant $K$ in \Cref{minsky} depends only on $\varepsilon$ and the topological complexity of $F$; i.e. for a fixed $\varepsilon$ and fixed genus, the constant is independent of the choice of pseudo-Anosov monodromy, and hence the hyperbolic metric on $N$. Rather than depending on any notion of ``complexity'' of the surface $S$, the constant $K$ in \Cref{thma} and \Cref{thmb} depends only on $\varepsilon$ and the aforementioned capacity of $f$.
	
	\section{Background, terminology and notation}\label{bg}
		\subsection{Surfaces and subsurfaces}\label{subsec:surfaces}
	In this subsection, let $X$ be a hyperbolic surface with compact (possibly empty) boundary and finitely many ends, all non-planar. We discuss the underlying topology of the surface. All surfaces in this work are orientable. When $X$ is of finite type  we let $g(X)$ denote its number of genus and $b(X)$ its number of boundary components. We let $\chi(X)= 2-2g(X) - b(X)$ be its Euler characteristic and $\xi(X)=3g(X)+b(X)-3$ its complexity.  When the surface is disconnected its Euler characteristic and complexity are simply the sums of the Euler characteristics and complexities of each connected component. All finite-type surfaces within this work satisfy $\chi(X)\leq 0$ so that for any subsurface $Y\subset X$, $|\chi(Y)| \leq |\chi(X)|$. 
	
	We consider isotopy classes of essential, simple, closed loops and proper isotopy classes of essential, properly embedded, simple paths on $X$. Recall that a path or loop is \emph{simple} if it is an embedding of the circle or closed interval, and \textit{essential} if it is non-peripheral and non-nullhomotopic in the surface. We use the terms \textit{arcs} and \textit{curves} to refer to the proper isotopy classes of essential paths and loops, and reserve the terms \emph{paths} and \emph{loops} for a particular (smooth) embedding of the closed interval or circle. We refer to a disjoint union of loops as a \textit{multiloop} and a disjoint union of distinct curves as a \emph{multicurve}.
	  
	Let $\mathcal{AC}(X)$ denote the \textit{arc and curve graph} of $X$, the simplicial graph whose vertex set $\mathcal{AC}^{(0)}(X)$ consists of arcs and curves of $X$ with edges between vertices whose representatives realize the minimum intersection number on the surface (outside of sporadic cases, this corresponds to disjointness).
	The graph $\mathcal{AC}(X)$ is a metric space endowed with the path metric $d_{\mathcal{AC}(X)}$ in which all edges have unit length. The \textit{curve graph} $\mathcal{C}(X)$ is the induced subgraph of $\mathcal{AC}(X)$ whose vertex set is the collection of curves of $X$. When $X$ is of finite type both $\mathcal{AC}(X)$ and $\mathcal{C}(X)$ have infinite diameter and are quasi-isometric to one another \cite[Theorem 1.1]{MM99} \cite[Theorem 1.3]{KP10}. We refer the reader to \cite{MM99} and \cite{S05} for more information about this graph and its (coarse) geometry.
	Given a compact annulus $Y$, the definition of $\mathcal{AC}(Y)$ requires more care. We define $\mathcal{AC}(Y)$ as the graph whose vertices are essential arcs from one end of $Y$ to the other, modulo isotopies \emph{relative to the endpoints} in $\partial Y$, again with edges between vertices which admit disjoint representatives.\\
	
	We say a subset $Y\subset X$ is an \textit{essential subsurface} if it is compact, connected, not homeomorphic to a pair of pants and its boundary is a disjoint collection of essential loops. We let $X \bbslash Y$ denote the metric closure of $X-Y$, typically embedded as the complement of a sufficiently small tubular neighborhood $\mathcal{N}_\varepsilon(\partial Y)$. Mapping classes of $Y$ are represented by elements of $\Homeo^+(Y, \partial Y)$, the set of orientation-preserving homeomorphisms which restrict to the identity on $\partial Y$. We naturally embed $\Map(Y)$, the mapping class group of $Y$, into $\Map(X)$ by extending a representative of a mapping class of $Y$ to be the identity on $X-Y$. When $Y$ is compact we refer to a mapping class $g\in \Map(Y)$ as \textit{fully supported} if $g$ acts loxodromically on $\mathcal{AC}(Y)$.  When $Y$ is an annulus each non-trivial element (i.e. some power of a Dehn twist) of $\Map(Y)$ is fully supported. For non-annular surfaces, fully supported mapping classes are \textit{pseudo-Anosov}. In fact, Masur--Minsky prove the following:
	
	\begin{theorem}[{\cite[Lemma 3.1]{MM99}}, c.f. {\cite[Theorem 2.10]{Mangahas13}}]\label{mintrans}
		Let $Y$ be a compact connected surface with $\chi(Y) \leq 0$. Then there exists a constant $c=c(Y)$, depending only on the topological type of $Y$, so that for any fully supported $g\in \Map(Y)$ and $\alpha \in \AC^{(0)}(Y)$ we have: 
			$$ c|n| \leq d_{\mathcal{AC}(Y)}(\alpha, g^n (\alpha)).$$	
	\end{theorem}
	
	Let $\tau_Y(g)$ denote the\textit{ stable translation length} of $g$ on $\mathcal{AC}(Y)$, i.e.:$$\tau_Y(g) := \lim_{n\to\infty}\frac{d_{\mathcal{AC}(Y)}(\alpha, g^{n}(\alpha))}{n}.$$ Recall that the definition is independent of the choice of $\alpha$ and that $\tau_Y(g) \leq \inf d_Y(\alpha, g(\alpha))$, where the infimum is taken over all $\alpha \in \mathcal{AC}^{(0)}(Y)$.
	It follows from the theorem above that when $g\in\Map(Y)$ is fully supported, $\tau_Y(g)\geq c$. 

\subsection{Subsurface projections}
	Following \cite{MM00}, given an essential $Y\subset X$ we define the \textit{subsurface projection map} $$\pi_Y: \mathcal{AC}(X) \to \mathcal{P}(\mathcal{AC}(Y))$$ from $\mathcal{AC}(X)$ to $\mathcal{P}(\mathcal{AC}(Y))$, the power set of $\mathcal{AC}(Y)$, by the following. 
	Consider the cover $p:{\tilde X_Y}\to X$ corresponding to the subgroup $\pi_1(Y)\leq \pi_1(X)$ and its compactification $X_Y= \bar{\H^2} / \pi_1(Y)$ homeomorphic to $Y$. Given any representative $\alpha \in \mathcal{AC}^{(0)}(X)$, its lift $p^{-1}(\alpha) \subset {Y}$ is a disjoint collection of paths and loops, and we define $\pi_Y(\alpha)$ to be the collection of components of $p^{-1}(\alpha)$ which represent (essential) arcs and curves of $X_Y$. Whenever the projection is non-empty, $\diam_{\mathcal{AC}(Y)}(\pi_{Y}(\alpha)) \leq 1$. This provides a coarsely well-defined map from $\mathcal{AC}(X) \to \mathcal{AC}(Y)\cup \{\emptyset\}$. Note that we may realize $\alpha$ and $\partial Y$ in minimal position by considering their geodesic representatives $\alpha^*$ and $\partial Y^*$ (see \Cref{fig:subprojection}).  When $Y$ is a non-annular subsurface of $X$ this produces a subsurface $Y^*$ with totally geodesic boundary $\partial Y^*$. In this setting, $\pi_Y(\alpha)$ is the set of vertices in $\mathcal{AC}(Y)$ comprising of each component of $Y^* \cap \alpha^*$ which represents an essential curve or arc in $Y^*$.
	
	\begin{figure}
	\centering 	\includegraphics[width=.6\textwidth]{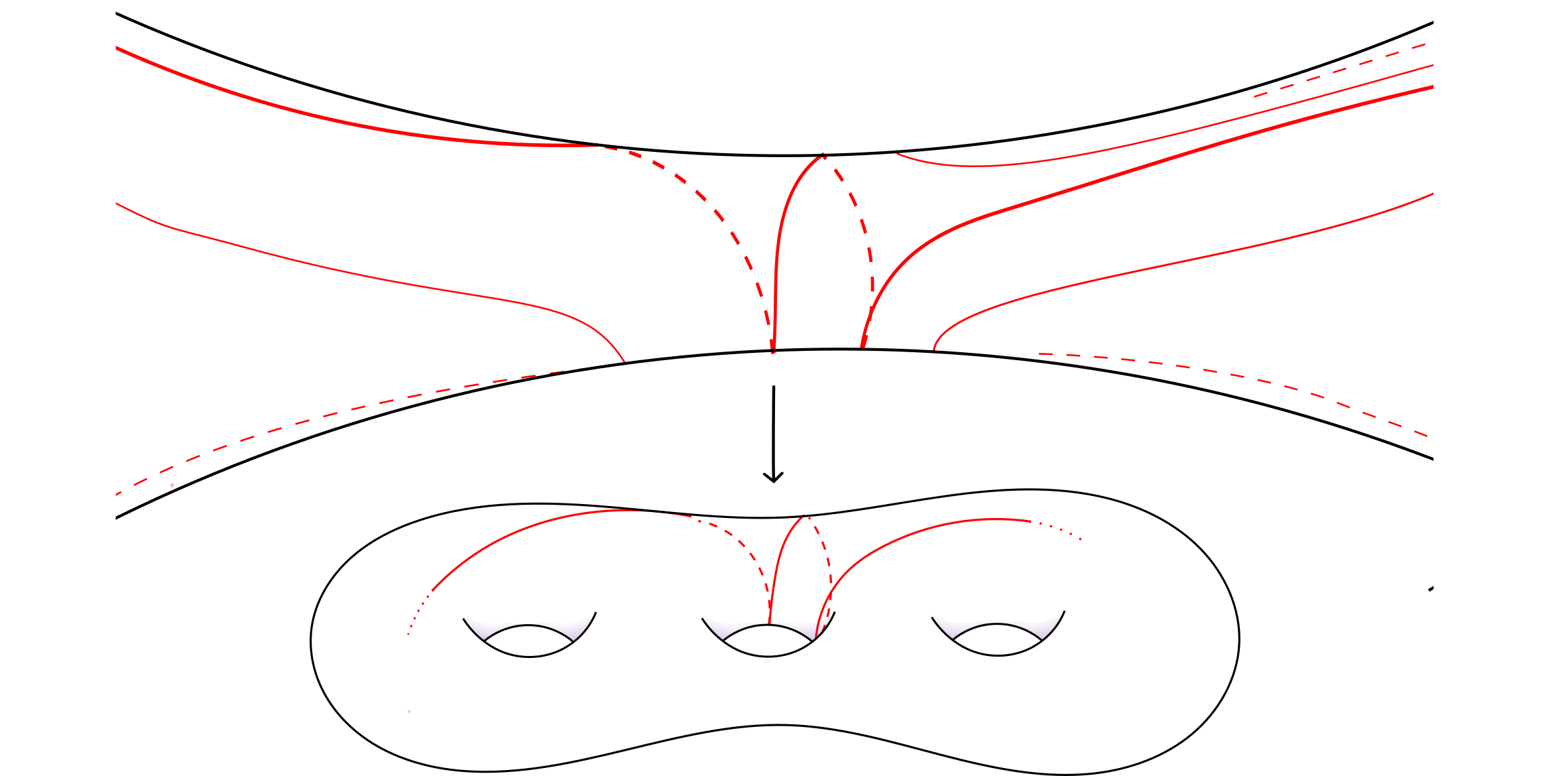}
			\caption{The lift of a path in $X$ to an annular cover $p: \tilde X_Y \to X$. While the preimage of the path consists of many arcs, only one represents a vertex of $\mathcal{AC}(Y)$.\label{fig:subprojection}}
	\end{figure}
	
	We will say that a subset $A\subset \mathcal{AC}(X)$ and $Y$ are \emph{disjoint} if $\pi_Y(A)=\emptyset$, and that they \emph{meet} otherwise. If $A$ is disjoint from $Y$ we will typically choose representative loops and paths which realize it as such. We note that for any $A \subset \mathcal{AC}(X)$ and isotopic surfaces $Y$, $Y'$, we have that $Y$ meets $A$ if and only if $Y'$ meets $A$. 
	
	Given subsets $A, B \subset \mathcal{AC}(X)$, we define their \textit{subsurface projection distance in $Y$} as: 
	$$ d_Y(A,B) := \text{diam}_{\mathcal{AC}(Y)}(\pi_Y(A) \cup \pi_Y(B)).$$ 
	
	Naturally, if $A$ is disjoint from $Y$, but $B$ meets $Y$, then $d_Y(A,B)=\diam(\pi_Y(B))$. For convenience, if both $A$ and $B$ are disjoint from $Y$, we set $d_Y(A, B)=0$. We list a few observations about subsurface projection distances which will be helpful for calculations. 
	
	\begin{remark}\label{projcalculations}
		Let $A= A_1,\ldots, A_n$ be subsets of $\mathcal{AC}(X)$ with $\pi_Y(A_i) \neq \emptyset$.  
		\begin{enumeri}
			\item $$\displaystyle d_Y(A_1,A_n) \leq \sum_{i=1}^{n-1}d_Y(A_i, A_{i+1}).$$
			\item For any $h \in \Map(X)$, $d_Y(A_1,A_2) = d_{h(Y)}(h(A_1),h(A_2))$. So, if $h$ is supported on $Y$, then $d_Y(A_1,A_2) = d_{Y}(h(A_1),h(A_2))$. 
			\item  If $h$ is supported on $Y$ then for any subset $A \subset \mathcal{AC}(Y)$ with $\pi_Y(h(A))\neq \emptyset$, 
			\[\tau_Y(h)\leq d_Y(A, h(A)).\]
			\item For subsets $A'_i\subset A_i$,
			\[ d_Y(A'_1, A'_2) \leq d_Y(A_1, A_2).\]
			\item If there exists a subsurface $W$ for which $\pi_W(A)$, $\pi_W(\partial Y) \neq \emptyset$ and $d_W(A, \partial Y)\geq 2$, then $\pi_Y(A) \neq \emptyset$.  
		\end{enumeri}
	\end{remark}
	
	Note that $\pi_Y(\alpha) \neq \emptyset$ does not imply $\pi_Y(h(\alpha)) \neq \emptyset$ for all $h\in \Map(X)$. In light of this, we apply (v) to ensure the subsurface projections are non-empty during relevant calculations.

\subsection{Foliations and laminations of surfaces}\label{laminatioNs}
	A \emph{geodesic lamination} $\Lambda$ of $X$ is a closed subset of $X$ which is the union of complete simple geodesics known as \emph{leaves} of the geodesic lamination. We are mostly concerned with the underlying topology of $\Lambda$ rather than any local structure, so we will simply refer to a closed subset of $X$ as a lamination if it is topologically equivalent to a geodesic lamination. 
	
	Given $\mathcal{G}$, a disjoint union of complete geodesics on $X$, we let $\bar{\mathcal{G}}$ denote its closure as a subset in $X$; its closure $\bar{\mathcal{G}}$ is (the support of) a geodesic lamination of $X$ (see \cite[Lemma 3.2]{CB88}).
	
	For an essential subsurface $Y$ and geodesic lamination $\Lambda$ of $X$, we may again lift the lamination to $\tilde\Lambda_Y \subset \tilde{X}_Y$ and define $\pi_Y(\Lambda)$ to be the (possibly empty) collection of vertices of $\mathcal{AC}(Y)$ which represent (essential) arcs and curves of $\tilde\Lambda_Y$. We similarly say that a lamination $\Lambda$ and subsurface $Y$ \textit{meet}  if $\pi_Y(\Lambda)\neq\emptyset$ and that they are disjoint otherwise. Notice this convention essentially ignores leaves of the lamination which spiral around some closed geodesic, as any such arc is non-proper. This will not be a problem for the Handel--Miller laminations (see \Cref{piY-HM}).

	\subsubsection{From foliations to laminations}
		A singular foliation $\mathcal{F}$ of a surface is a decomposition of $X$ into $1$-manifolds known as \textit{leaves} and a finite collection $\mathcal{P}$ of \textit{singularities} so that the collection $\mathcal{F}-{\mathcal{P}}$ forms a foliation of $X - \mathcal{P}$. A leaf which is disjoint from the singular collection $\mathcal{P}$ is a \textit{regular} leaf, while one which meets $\mathcal{P}$ is called \textit{singular}. For a singular foliation, a \textit{leaf line} of $\mathcal{F}$ is either a regular leaf of $\mathcal{F}$ or a bi-infinite line formed by a union of at least two singular leaves so that any pair of leaves which meet at a singularity are adjacent. A closed subset $\Lambda$ of $\mathcal{F}$ formed by the union of leaf lines of $\mathcal{F}$ is known as a \textit{(singular) sublamination} of $\mathcal{F}$. We use the same terminology of leaf, leaf line, and singular leaf to refer to the leaves of the sublamination. 
		
		A Reebless singular foliation $\mathcal{F}$ on a hyperbolic surface $X$ has a corresponding lamination $\Lambda({\mathcal{F}}) \subset X$ obtained by the following process of ``straightening'' leaf lines of $\mathcal{F}$ \cite[Theorem 1]{Lev83}. The lifts of each leaf line of $\mathcal{F}$ form a collection of disjoint bi-infinite quasi-geodesics of ${\H^2}$ which defines a subset $\partial^2(\mathcal{F})$ of $\partial^2 (\H^2)$, the space of unoriented geodesics of $\H^2$. For each point $(z,w) \in \partial^2(\mathcal{F}) $, consider the unique bi-infinite geodesic representative between $(z,w)$. The union of such geodesics is a $\pi_1(X)$-invariant lamination of ${\H^2}$ which descends to a lamination $\Lambda(\mathcal{F})$ of $X$.  
		
		For an essential subsurface $Y \subset X$, we define $\pi_Y(\mathcal{F}):= \pi_Y(\Lambda(\mathcal{F}))$. Similarly, if $\Lambda$ is a (singular) sublamination of $\mathcal{F}$, we define $\pi_Y(\Lambda)$ to be the union of $\pi_Y(\ell)$ over all leaf lines $\ell$ of $\Lambda$.
		
		We will later refer to singular $2$-dimensional foliations of $3$-manifolds, but only those obtained as the suspension of a singular foliation of a surface. For such singular foliations, the $1$-dimensional leaves of the surface foliation suspend to $2$-dimensional leaves and the finite set of singularities suspend to closed $1$-dimensional singular orbits, i.e. loops in the $3$-manifold. 

\subsection{End-periodic homeomorphisms}\label{eph}
	From this section onward, let $S$ be a boundaryless infinite-type surface with finitely many ends. We denote by $\Ends(S)$ the finite set of ends of $S$. We refer the reader to the survey paper of Aramayona--Vlamis \cite{AV20} for further exposition of surfaces of this type.
	A homeomorphism $f$ of $S$ is \textit{end-periodic} if for each end $\epsilon \in \text{Ends}(S)$ there exists $p\in \Z_{\geq 0}$ and $U_\epsilon$, a connected neighborhood of $\epsilon$ with compact essential frontier, so that the end $\epsilon$ is either \textit{attracting} or \textit{repelling} under $f$ in the following sense:
	
	\begin{itemize}
		\item ($\epsilon$ is an \textit{attracting} end): $f^p(U_\epsilon) \subset U_\epsilon$ is a proper inclusion and the collection $\{f^{pn}(U_\epsilon)\}_{n\in \Z}$ forms a neighborhood basis of $\epsilon$.
		\item ($\epsilon$ is a \textit{repelling} end): $f^{-p}(U_\epsilon) \subset U_\epsilon$ is a proper inclusion and the collection $\{f^{pn}(U_\epsilon)\}_{n\in \Z}$ forms a neighborhood basis of $\epsilon$.
	\end{itemize} 
	
	 Following terminology of Field--Kim--Leininger--Loving \cite{FKLL23}, we refer to a neighborhood $U_\epsilon$ which realizes $\epsilon$ as an attracting or repelling end as a \textit{nesting neighborhood} of $\epsilon$. The action of $f$ partitions $\Ends(S)$ into  the disjoint union $\Ends ^+(S) \,\sqcup\, \Ends^-(S)$, where $\Ends^+(S)$ denotes the set of attracting ends and $\Ends^-(S)$ denotes the set of repelling ends. $\Ends^\pm(S)$ further decomposes into a collection of \textit{end orbits} $\mathcal{O}_1, \ldots, \mathcal{O}_k$. The period of an end $\epsilon$ is the cardinality of its end orbit; or equivalently, it is the minimal positive $p$ for which $f^p (\epsilon)=\epsilon$. We will use $+$ (respectively  $-$) to denote an attracting (resp. repelling) end, one of its neighborhoods, or a curve in such a neighborhood.  We typically order the ends in an attracting orbit so that $\epsilon_i =f^{i}(\epsilon_0)$ and in a repelling orbit so that $\epsilon_i =f^{-i}(\epsilon_0)$.
	 
	As in \cite{F97}, we organize nesting neighborhoods of $\Ends^\pm(S)$ and define a positive (resp. negative) \emph{ladder} of $f$ as a neighborhood of $\Ends^\pm(S)$ of the form 
	\[U_+ = \bigcup_{\epsilon \in \Ends^+(S)} U_\epsilon \quad \text{ and } \quad U_- = \bigcup_{\epsilon \in \Ends^-(S)} U_\epsilon,\]
	where $\{U_\epsilon\}_{\epsilon\in\Ends^\pm(S)}$ are pairwise disjoint and each $U_\epsilon$ is a nesting neighborhood of $\epsilon$. We say a positive (resp. negative) ladder $U_\pm$ is \textit{tight} if $f^{\pm}(U_\pm) \subsetneq U_\pm$ is a proper inclusion. Observe that each end-periodic map admits a tight ladder. To see this, let $U$ be an attracting nesting neighborhood of an end of period $p$. Choose a collection $j_1, \ldots, j_{p-1}$ of distinct separating loops of $U \bbslash f^p(U)$ so that each separates $\partial U$ from $\partial f^p(U)$ and that $j_{i+1}$ is contained in the non-compact component of $U_p \bbslash j_i$. One can verify the union $U_0= U \cup f^{-1}(U_1) \cup f^{-2}(U_2) \ldots \cup f^{-(p-1)}(U_{p-1})$ properly contains its image. Here, $U_i$ is the component of $U-  j_i$ which is a nesting neighborhood of $\epsilon$ under $f$. \\
	
	We use ladders to define the \textit{positive escaping set} and \textit{negative escaping set} of $f$ respectively:
	\[\mathcal{U}_+ = \bigcup_{n\in\Z} f^{n}(U_+) \quad \text{ and } \quad \mathcal{U}_- = \bigcup_{n\in\Z} f^{n}(U_-).\] 
	We also refer to the \textit{escaping set of an end} $\epsilon$ of period $p$ by: \[\mathcal{U}_\epsilon = \bigcup_{n\in\Z} f^{pn}(U_\epsilon),\] where $U_\epsilon$ is a nesting neighborhood of $\epsilon$. From the definition of end-periodicity one can verify that these definitions are independent of the choice of $U_\pm$ (or $U_\epsilon$). As the name implies, the escaping sets form the collection of points whose positive (resp. negative) iterates under $f$ eventually land in an attracting (resp. repelling) nesting neighborhood, hence ``escape'' the end. 
	
	We say that a loop is \textit{positively} (resp. \textit{negatively}) \textit{escaping} under $f$ if it is contained in the positive (resp. negative) escaping set of $f$. A curve is escaping under $f$ if some loop in its homotopy class is. The positive and negative ladders themselves are surfaces of infinite type, thus there are uncountably many distinct curves which escape into the positive or negative ends under iteration of the map $f$.
	
	\begin{lemma}[Field--Kim--Leininger--Loving {\cite[Lemma 2.4]{FKLL23}}]\label{boundarysurfaces} Let $f$ be an end-periodic homeomorphism. Then $f$ acts freely, properly discontinuously and cocompactly on $\mathcal{U}_\pm$. 
		The quotient of the action $\mathcal{U}_\pm / \langle f \rangle$ defines the following surfaces
		\[ \Sigma_+ :=\mathcal{U}_+ / \langle f \rangle \quad\text{ and }\qquad \Sigma_- :=\mathcal{U}_- / \langle f \rangle.\]
		Further, the surfaces satisfy $\chi(\Sigma_+)=\chi(\Sigma_-)$.
	\end{lemma}
	Observe that the components of $\Sigma_+ \cup \Sigma_-$ are in bijection with the end orbits of $f$. 
	
	\subsubsection{Types of end-periodic homeomorphisms}
		A loop or line $\alpha$ is \textit{periodic} under $f$ if there exists $m\neq 0$ s.t. $f^m (\alpha)$ is properly isotopic to $\alpha$. We say a loop or line $\alpha$ is \textit{reducing} if  $\alpha\subset\mathcal{U}_+ \,\cap\, \mathcal{U}_-$,  in other words, if $\alpha$ is both positively and negatively escaping. We say a curve or arc is periodic or reducible if one of its representatives is.
		
		An end-periodic homeomorphism is \textit{atoroidal} if it admits no periodic curves. As the name implies, an atoroidal end-periodic homeomorphism yields an atoroidal mapping torus $M_f$ (see \Cref{fkllhyp}). Field--Kim--Leininger--Loving introduce \emph{strongly irreducible} end-periodic homeomorphisms, which are atoroidal end-perioidic homeomorphisms which admit no periodic arc nor reducing curve. We direct the reader to \Cref{examples} for more examples of end-periodic homeomorphisms, including those arising from a construction of Field--Kim--Leininger--Loving.
		
	\subsubsection{Juncture orbits}
		We say that an essential multiloop in $\mathcal{U}_+$ (resp. $\mathcal{U}_-$) is a positive (resp. negative) \textit{juncture} of $f$ if it is the boundary of a tight
		positive (resp. negative) ladder. This is not an isotopy-invariant property; a multiloop need not bound a tight ladder of a given homeomorphism just because another loop in its homotopy class does. In light of this, we emphasize that we are often concerned with the genuine representative loop rather than its isotopy class. \\
		
		Fix a positive jucture $j_+$ and negative juncture $j_-$ of $f$. We define the \emph{positive juncture orbit} $J_+$ of $j_+$ and \emph{negative juncture orbit} $J_-$ of $j_-$ as the multiloops
		$$J_+ = \bigcup_{k\in\Z} f^{k}({j_+}) \quad\text{ and }\quad J_- = \bigcup_{k\in\Z} f^{k}({j_-}).$$ To verify that these unions are indeed multiloops, consider a tight ladder $U_+$ bounded by $j_+$. It follows that $\forall k\in \Z_{> 0}, f^k(U_+) \subset U_+$ is a proper inclusion, which implies that each component of $f^i(j_+)$ is either disjoint from or coincides with some component of $f^k(j_+)$ for all $i,k$ $\in \Z$. 
		
		\begin{remark}
			Consider the partial juncture orbits
			\begin{align*}
			 	\jpm &= \bigcup_{i\leq k} f^i(j_+) & \jpp &= \bigcup_{i\geq k} f^i(j_+)\\ 
			 	\jmp &= \bigcup_{i\geq k} f^i(j_-) & \jmm &= \bigcup_{i\leq k} f^i(j_-).
			\end{align*}
			
			The partial juncture orbits $\jpp$ and $\jmm$ are both closed subsets of $S$. 
		\end{remark}
		
		While the juncture orbits are somewhat uninteresting in the ends of $S$, their closures display more curious behavior.
	
\subsection{Handel--Miller laminations}\label{HMlam}

	Fix a hyperbolic metric $X$ on $S$ and let $J_+$, $J_-$ be a pair of positive and negative juncture orbits for $f$. 
	Let $\mathcal{J}_\pm$ be the union of geodesic representatives for each curve in $J_\pm$. Since $J_\pm$ is a multiloop, ${\mathcal{J}_\pm}$ is foliated by geodesics and defines a geodesic lamination $\bar{\mathcal{J}_\pm}$. The \emph{Handel--Miller laminations} \cite[Section 4]{CCF21} are the geodesic laminations defined as
	 
	$$\Lambda^+ = \bar{\mathcal{J}_-}- {\mathcal{J}_-}$$ 
	$$\Lambda^- = \bar{\mathcal{J}_+}- {\mathcal{J}_+}.$$ 
	
	When non-empty, it is helpful to consider the Hausdorff distance between the junctures and the respective laminations. Let $\mathcal{J}_+^k$ denote the geodesic representative of the multiloop  $\jpm$ and let $\mathcal{J}_-^{k}$ denote the geodesic representative of $\jmp$. For every $\varepsilon > 0$, there exists $K$ for which: 
	$\mathcal{J}_+^K \subset \mathcal{N}_\varepsilon (\Lambda^-)$ and $\Lambda^- \subset \mathcal{N}_\varepsilon(\mathcal{J}_+^K)$ (likewise,  $\mathcal{J}_-^K\subset N_\varepsilon (\Lambda^+)$ and $\Lambda^+ \subset \mathcal{N}_\varepsilon(\mathcal{J}_-^K)$).  In this sense, the negative iterates of a positive juncture limit to the negative Handel--Miller lamination and the positive iterates of a negative juncture limit to the positive Handel--Miller lamination.
	
	\begin{figure}
		\centering 
		\includegraphics[width=.65\textwidth]{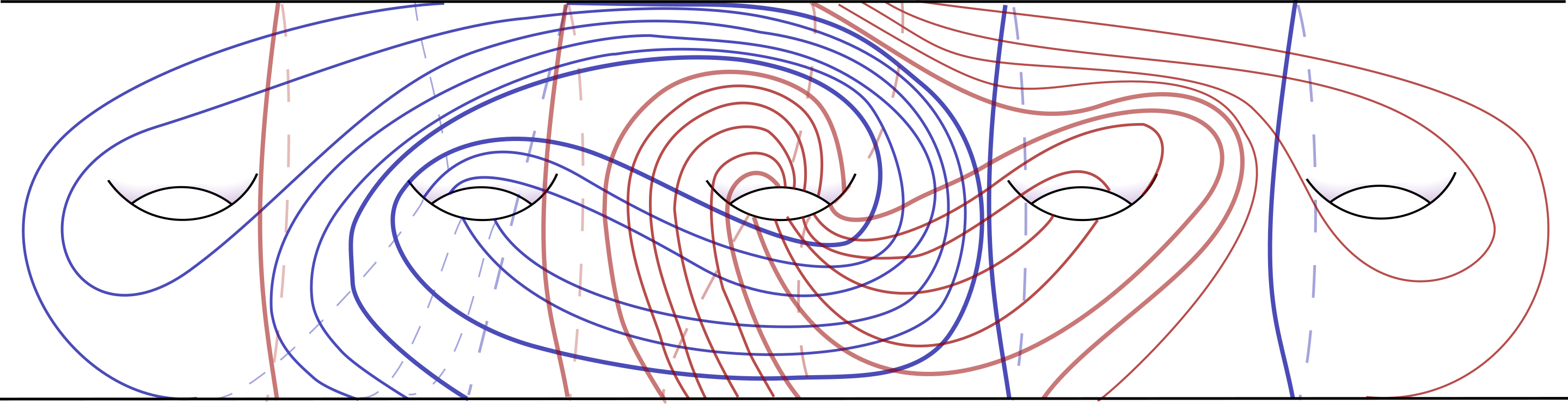}
		\caption{A few components of $J_+$ (in red) and $J_+$(in blue) of an atoroidal end-periodic homeomorphism.}
	\end{figure}
	
 	After fixing a hyperbolic metric $X$ of $S$, the Handel--Miller geodesic laminations are independent of the isotopy class of $f$ as well as the choice of juncture orbits. Any two choices of juncture orbit pairs $J_-, J_+$ and $J'_-, J'_+$ yield the same Handel--Miller laminations $\Lambda^+$, $\Lam^-$ \cite[Corollary 4.72]{CCF21}. Note that $\Lambda^+$, $\Lambda^-$ are empty if $\mathcal{J}_-$, $\mathcal{J}_+$ themselves are closed---Cantwell--Conlon--Fenley show that this is the case precisely when the map is a \textit{translation}, i.e. there exists a nesting neighborhood $U_\epsilon$ of an end $\epsilon$ so that $\mathcal{U}_\epsilon = \bigcup_{n\in\Z} f^n (U_\epsilon)= S$ \cite[Corollary 4.78]{CCF21}. 
 
 	After fixing a hyperbolic structure $X$, there is a representative homeomorphism $f^\sharp$ known as a \textit{Handel--Miller representative} of $f$ which is isotopic to $f$ and preserves the Handel--Miller geodesic laminations (and a choice of geodesic juncture orbits). We point the reader to work of Fenley \cite{F97} and Cantwell--Colon--Fenley \cite{CCF21} for further exposition.  
 
	 \begin{observation}\label{disjoint:lamjunc}
		 	For any positive juncture orbit, its geodesic multicurve $\mathcal{J}_+$ is disjoint from the negative Handel--Miller lamination $\Lambda^-$. Likewise, for any negative juncture orbit, its geodesic multicurve $\mathcal{J}_-$ is disjoint from the positive Handel--Miller lamination $\Lambda^+$.	
		 	
		 	Furthermore, for any ${U}^*_+$, a neighborhood of $\Ends^+(S)$ bounded by the geodesic representative of a juncture ${j}_+$, $\Lambda^- \,\cap\, {U}^*_+ = \emptyset$. Similarly, for any negative juncture $j_-$, the corresponding neighborhood ${U}^*_-$ is disjoint from $\Lambda^+$. 
	 \end{observation}
		
	The Handel--Miller laminations detect escaping curves in the following sense:  
	
	\begin{proposition}\label{escaping}
		Suppose $f$ is atoroidal and let $\gamma$ be a curve of $S$. The following are equivalent: 
			\begin{enumeri}
				\item $\gamma$ is positively (resp. negatively) escaping under $f$.
				\item $\gamma$ meets only finitely many components of any positive juncture orbit $J_+$ (resp. negative juncture orbit $J_-$).
				\item $\gamma$ is disjoint from the negative  Handel--Miller lamination $\Lambda^-$ (resp. positive Handel--Miller lamination $\Lambda^+$).
			\end{enumeri}
	\end{proposition}	
	
	\begin{proof}
		We prove the statement for a representative of a positively escaping curve as the proof for a negatively escaping curve is analogous. Consider a representative $\gamma \subset \mathcal{U}_+$ and let $J_+ = \bigcup f^k(j_+)$ be any positive juncture orbit of $f$. We are pleased if $\gamma$ is disjoint from $J_+$ and otherwise we can isotope the representative and juncture orbits so that $\gamma$ meets $J_+$ transversely and has finite intersection with each component of $J_+$. If a sequence $\{x_n\}$ in $J_+\cap \gamma$ converges then it must eventually be contained in a single component of $J_+$. This implies that $J_+\cap \gamma$ must be finite, as otherwise would produce an infinite intersection between $\gamma$ and some component of $J_+$. It follows that there are only finitely many $i$ for which $f^i(j_+)$ meet $\gamma$. So, $\gamma$ is disjoint from $\jpm$ for sufficiently negative $k$ and therefore also from the limiting lamination $\Lambda^-$. This shows $(i) \implies (ii) \implies (iii)$.
		
		We now show that $(iii) \implies (i)$, assuming $\Lambda^-$ is non-empty (otherwise, the statement is vacuously true and all curves are positively escaping). If $\gamma$ is disjoint from $\Lambda^-$, then there exists $\varepsilon>0$ for which a representative of $\gamma$ is disjoint from $\mathcal{N}_\varepsilon(\Lambda^-)$. Thus, for some $k_0$ we have that $J_+^{\leq k_0}$ is disjoint from $\gamma$.
		Each $f^{i}(j_+)$ defines a positive ladder $U_i$ and we argue that $\gamma$ must be contained in one of them. Suppose otherwise and let $C=S - (U_+ \sqcup U_-)$ be a core containing $\gamma$ so that $\partial_+ C \subset J_+$. For all $k\geq\max\{-k_0,0\}$, $f^{k}(\gamma)$ is disjoint from both $U_-$ (by tightness and the nesting property of $f$) and $U_+$ (by our assumption). So for all such $k$, $f^k(\gamma)$ is contained in the core.
		Since the space of geodesic laminations on $C$ is compact, some subsequence of $\{(f^k(\gamma))^*\}_{k\geq 0}$ converges to an $f$-invariant (up to isotopy) geodesic lamination $\lambda_\gamma$. 
		Let $Y_\gamma \subset C$ be the minimal surface with geodesic boundary containing $\lambda_\gamma$. Then $Y_\gamma$ must also be $f$-invariant  as its supporting lamination is. This implies the finite boundary curves of $Y_\gamma$ are permuted by $f$, contradicting that $f$ is atoroidal.
		Thus, $\gamma$ must be contained in some positive ladder.
	\end{proof}
	
	\begin{lemma}\label{piY-HM}
		For any essential subsurface $Y$, 
		\begin{enumerate}[label=(\arabic*)]
			\item  $\pi_Y(\Lambda^\pm) \in \mathcal{P}(\mathcal{AC}(Y))$ is either empty or a diameter-$1$ subset of $\mathcal{AC}(Y)$. 
			\item If $\pi_Y(\Lambda^\pm) = \emptyset$, then $Y$ is isotopic to a surface $Y'$ with $Y' \cap \Lambda^{\pm} = \emptyset$.
			\item The following are equivalent: 
			\begin{enumerate}[label=(\roman*)]
				\item Both laminations $\Lambda^+$ and $\Lambda^-$ meet $Y$ 
				\item Both laminations $\Lambda^+$ and $\Lambda^-$ meet $\partial Y$ 
				\item $Y$ is not contained in any positive or negative ladder of $f$.  
			\end{enumerate}	
		\end{enumerate}
	\end{lemma}
	
	\begin{proof}
		The first conclusion follows immediately as the leaves of a lamination are pairwise disjoint. 
		Recall that $\pi_Y(\Lambda^\pm)$ consists of the essential components of the lift $\tilde{\Lam}^\pm_Y$ of $\Lam^\pm$ under the covering map $p: \tilde{X}_Y \to X$. Consider the convex core $\tilde C_Y$ of $\tilde X_Y$. When $Y$ is non-annular $\tilde C_Y$ is homeomorphic to $Y$, while when $Y$ is an annulus $\tilde C_Y$ is the core geodesic in $\tilde X_Y$. The core $\tilde C_Y$ is disjoint from all proper inessential leaves of $\tilde{\Lam}^\pm$, as otherwise yields a bigon between some component of $\partial\tilde C_Y$ (or the core geodesic in the annular case) and the inessential leaf. In the annular case, we take $\tilde Y'$ to be a sufficiently small $\varepsilon$-neighborhood of $\tilde C_Y$ which still avoids each inessential proper arc of $\tilde{\Lam^\pm}$ and to whom the restriction of $p$ is a homeomorphism. In the non-annular case, we let $\tilde Y'$ be the subsurface of $\tilde C_Y$ with the same properties above, obtained by removing a sufficiently small half $\varepsilon$-neighborhood of $\partial \tilde C_Y$ from $\tilde C_Y$.
		
		Now suppose $\pi_Y(\Lam^\pm)=\emptyset$. To ensure the surface $Y'$ above is disjoint from the lamination, we need only verify that every leaf $\tilde l$ of $\tilde \Lam^\pm_Y$ is proper. First suppose that $Y$ is non-annular. If the leaf $\tilde l$ is disjoint from $\tilde C_Y$ then it is inessential and avoids $\tilde Y'$, so we assume $\tilde l$ meets $\tilde C_Y$. If both ends of $\tilde l$ escape $\tilde C_Y$ (i.e., each end of $\tilde l$ has a neighborhood which is contained in $\tilde X_Y - \tilde C_Y$), then $\tilde l$ is proper. Otherwise, some half-leaf of $\tilde l$ is entirely contained in $\tilde C_Y$. This is ruled out by a corollary of Cantwell--Conlon--Fenley \cite[Corollary 4.50]{CCF21} which states that each half-leaf escapes arbitrarily far into the ends of $X$, so a lift cannot be contained in a bounded region of the cover $\tilde X_Y$. So each $\tilde l$ is proper. 
		The case when $Y$ is an annulus is similar. If the leaf $\tilde l$ is disjoint from $\tilde C_Y$, the core geodesic of $\tilde X_Y$, it is either inessential or spirals around $\tilde C_Y$. The latter case is ruled out by the above corollary, as that would imply some half-leaf of $\tilde l$ is entirely contained in $\tilde Y'$. 	
		Taking $Y'= p(\tilde Y')$ completes the proof of (2).
		
		Since the laminations $\Lambda^\pm$ contain no closed leaves \cite[Corollary 4.47]{CCF21}, and by the discussion above, each essential lift of a leaf of $\Lambda^\pm$ must be a properly embedded arc in $\tilde X_Y$. Thus, if $\pi_Y(\Lambda^\pm)\neq \emptyset$, the leaves must meet $\partial Y$. \Cref{escaping} confirms that if $\Lambda^+$ and $\Lambda^-$ meet $\partial Y$, then $\partial Y$ is neither positively nor negatively escaping, which means it is not contained in any positive nor negative ladder of $f$. Similarly, if $Y$ avoids either lamination then it must be contained in a positive or negative ladder. This confirms (i) $\implies$ (ii) $\implies$ (iii) $\implies$ (i), thus the proof of (3).
	\end{proof}
	
	\begin{corollary}\label{mustMeetLam}
		Suppose $f$ is (strongly) irreducible. For any curve $\gamma \subset S$ either $\gamma$ meets $\Lambda^+$, $\gamma$ meets $\Lambda^-$, or $\gamma$ meets both laminations. Similarly, for each essential subsurface $Y$ of $S$, either $Y$ meets $\Lambda^+$, $Y$ meets $\Lambda^-$ or $Y$ meets both. In particular, no curve nor subsurface can avoid both the positive and negative laminations. 
	\end{corollary}
	
	The subsurface projection coefficients provide a more explicit sense in which our original end-periodic homeomorphism $f$ preserves the Handel--Miller geodesic laminations: 
	
	\begin{lemma}[Pseudo $f$-invariance]\label{f-invariance}
		For any essential subsurface $Y$,
		\[\pi_Y(f^n(\Lambda^\pm)) = \pi_Y(\Lambda^\pm)\text{ and }\] \[d_Y(\Lambda^+, \Lambda^-) = d_{f^n(Y)}(\Lambda^+, \Lambda^-).\]
	\end{lemma}
	\begin{proof}
		Let $\Lambda^+, \Lambda^-$ denote the Handel--Miller geodesic laminations and let $f^\sharp$ denote a Handel--Miller representative homeomorphism isotopic to $f$ so that $f^\sharp (\Lambda^\pm)=\Lambda^\pm$. We assume $\pi_Y(\Lambda^-), \pi_Y(\Lambda^+) \neq \emptyset$. Lift $\Lambda^\pm$ and $f(\Lambda^\pm)$ to laminations $\widetilde{\Lambda^\pm}$ and $\widetilde{f(\Lambda^\pm)}$ on the cover $\tilde X_Y$ corresponding to $Y$. The isotopy between the lifts $\tilde f$ and $\tilde{f^\sharp}$ implies that for each leaf $l$ of $\Lam^\pm$, the geodesic representative of $\widetilde{f(l)}$ is exactly the geodesic leaf  $\tilde{f^\sharp(l)}$ of $\widetilde{\Lambda^\pm}$. It follows that $\pi_Y(f^n(\Lambda^\pm)) = \pi_Y(f^{\sharp n}(\Lambda^\pm))= \pi_Y(\Lambda^\pm)$ so that $$d_{f^nY}(\Lambda^+, \Lambda^-) = \diam_{\mathcal{AC}(f^nY)}(\pi_Y(f^n(\Lambda^+)) \cup \pi_Y(f^n(\Lambda^-)))= d_Y(\Lambda^+, \Lambda^-).$$
	\end{proof}
	
	\subsubsection{Cores}\label{subsec:cores}
		A pair of disjoint positive and negative tight ladders, $U_+$, $U_-$, defines a (not necessarily connected) compact surface $C= S - (U_+ \sqcup U_-)$ called a \emph{core}. The boundary of the core consists of positive and negative junctures denoted respectively by $\partial_\pm C = \partial U_\pm$. A core $C'$ is a \textit{subcore} of $C$ if $C' \subset C$. Since cores bound positive and negative tight ladders, we may expand a core by the action of $f$. We define a core $C$ to be $f$-\textit{extended} if it contains a subcore $C'$ so that $$\partial_+ C' \subset \bigcup_{k>0} f^{-k}(\partial_+ C), \quad \text{and}\quad \partial_- C' \subset \bigcup_{k>0} f^{k}(\partial_- C).$$ 
		
		\begin{figure}
		\centering 
			\begin{center}
			\includegraphics[width=.6\textwidth]{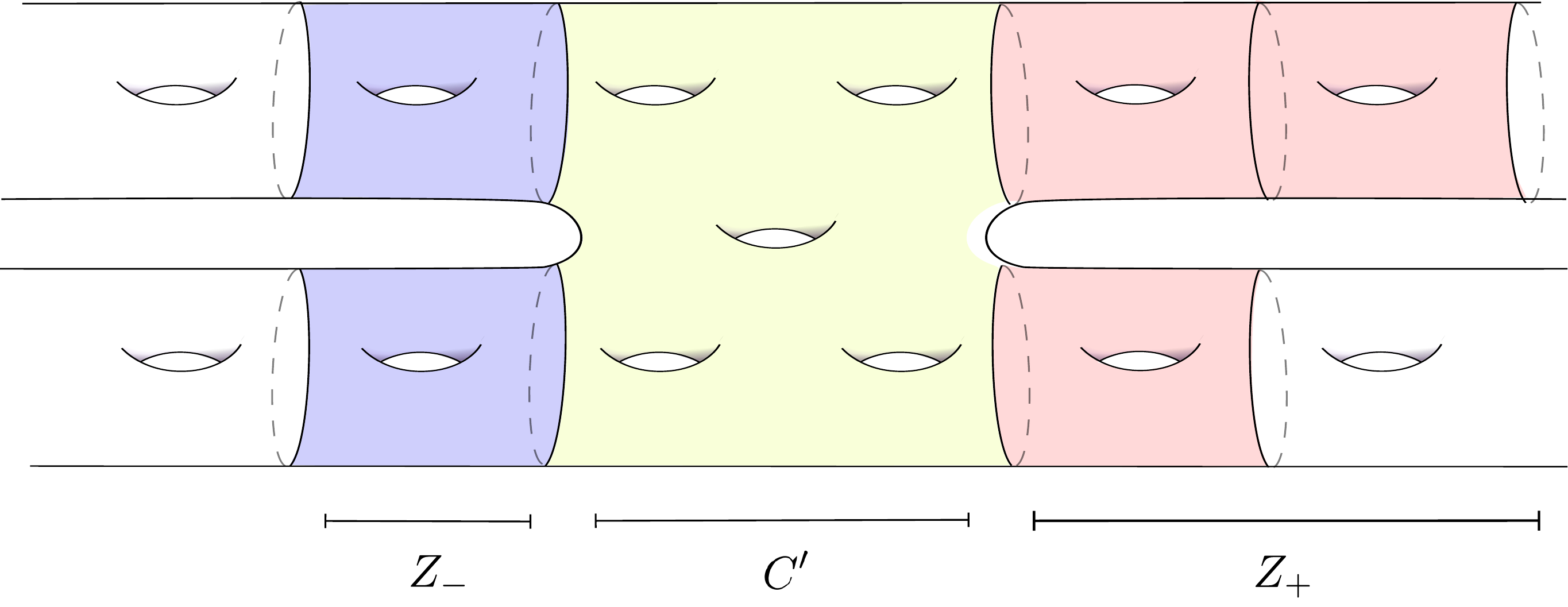}
			\caption{The core $C = Z_-\cup C' \cup Z_+$ is $f$-extended.}
			\end{center}
		\end{figure}
		
		Field--Kent--Leininger--Loving \cite[Section 2]{FKeLL23} use cores and shifting behavior of an end-periodic map $f$ to give a notion of complexity for the map. 
		A core $C$ is \textit{minimal} if $|\chi(C)| \leq |\chi(C')|$ for all cores $C'$ of $f$. The \emph{core characteristic} $\chi(f)$ of $f$ is the absolute value of the Euler characteristic of a minimal core\footnote{Our definition of $\chi(f)$ is the negative of that defined by Field--Kent--Leininger--Loving; this is just a cosmetic change.}, while the \textit{end complexity} $\xi(f) := \xi(\Sigma_+ \cup \Sigma_-).$ We direct the reader's attention to \Cref{boundarysurfaces} which states that $\partial \Mf$ is homeomorphic to the compact subsurface $\Sigma_+ \cup \Sigma_-$.

\subsection{The compactified mapping torus}\label{mfconstruct}
	We review Field--Kim--Leininger--Loving's discussion of the compactification of $M_f$ which was originally introduced by Fenley \cite[Section 3]{F97}.
	Recall that $M_f$ is obtained by the quotient map $$\pi: S \times\R \to (S \times\R)/\langle F \rangle,$$ where $F(x,t) = (f(x), t-1)$.
	Let $\tilde{M}_\infty$ be the partial compactification of $S\times \R$ given by $\tilde{M}_\infty= S\times \R \, \sqcup \, \mathcal{U}_{+\infty} \, \sqcup \, \mathcal{U_{-\infty}}$, where $\mathcal{U_{\pm \infty}}= \mathcal{U_\pm}\times \{\pm \infty\}$. The partial compactification admits a foliation $\tilde{\mathcal{F}}$ consisting of leaves of the form $S_t = S\times \{t\}$, $\mathcal{U}_{+\infty}$, and $\mathcal{U}_{-\infty}$. The action of $F$ on $S\times\R$ extends to $\widetilde{M}_\infty$ so that $F(x,\pm \infty) = (f(x),\pm \infty)$ and preserves the leaves of $\widetilde{\mathcal{F}}$. 
	The compactification of $M_f$ is the $3$-manifold with boundary \[\Mf= \tilde{M}_\infty/\langle F \rangle.\] Let $\mathcal{F}$ denote the corresponding $2$-dimensional depth-one foliation of $\Mf$ obtained by $\tilde{\mathcal{F}}/\langle F \rangle$. Each component $\Sigma$ of $\partial M$ is a depth-zero leaf of $\mathcal{F}$, while the $S_t$ leaves are depth-one.
	
	\begin{proposition}[Field--Kim--Leininger--Loving {\cite[Proposition 3.1]{FKLL23}}]\label{cpctification}
		$\bar{M}_f$ is a compact $3$-manifold with boundary whose interior is homeomorphic to $M_f$.  Further, $\partial \Mf = \partial_+ \Mf \,\sqcup\, \partial_- \Mf $ where $\partial_\pm \Mf=\Sigma_\pm$.
	\end{proposition}
	
	\begin{figure}
	\centering 
		\includegraphics[width=.6\textwidth]{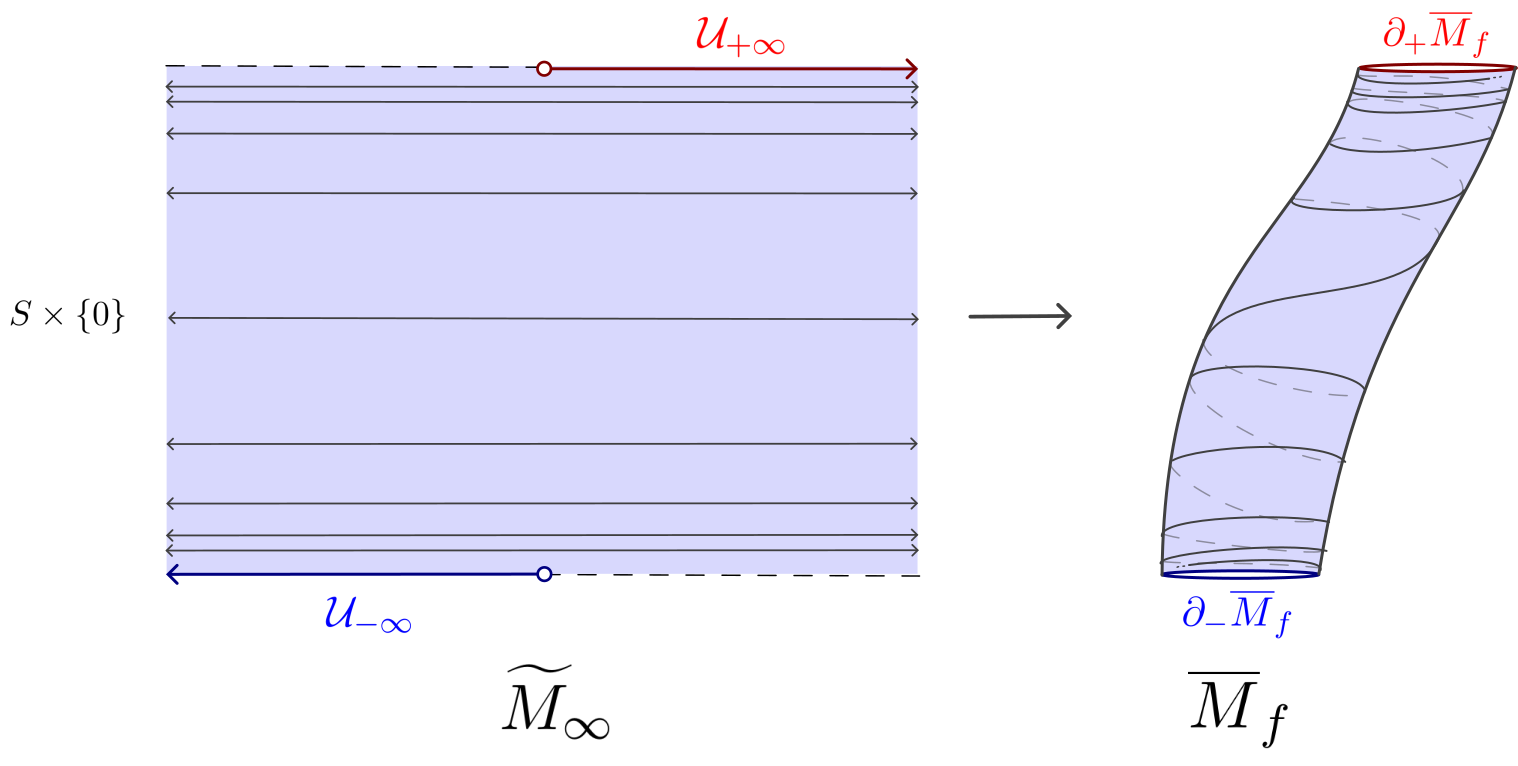}
		\caption{A schematic of the quotient $\Mf = \tilde M_\infty / \langle F \rangle$. The depth-one foliation $\widetilde{\mathcal{F}}$ descends to a depth-one foliation $\mathcal{F}$ of $\Mf$, where the leaves of the form $S_t= S\times \{t\}$ spiral into the compact leaves.}
	\end{figure}
	
	The vertical flow on $S\times \R$ given by $\hat\varphi((x,s),t) =(x,s+t)$ extends to a semi-flow on $\tilde{M}_\infty$ by adding endpoints $(x,\pm\infty)$ to each flow line $\ell_x(t) = (x,t)$ of $x\in\mathcal{U}_\pm$. 
	This vertical flow and semi-flow are both $F$-invariant and descend to a flow and semi-flow on $M_f$, known as the \textit{suspension (semi)-flow} $\varphi_f$ on $\Mf$. We recycle notation and use $\varphi_f$ to refer to both the flow on $M_f$ and the semi-flow on $\Mf$. Let $\eL$ be the oriented  $1$-dimensional foliation induced by the suspension semi-flow lines of $\varphi_f$ on $\Mf$.
	
	\begin{proposition}[Product structure at the boundary]\label{product}
		In a product neighborhood $U \cong \partial_+ \Mf \times (-\infty, \infty]$ of $\partial_+ \Mf$, the foliation $\eL$ is 
		modeled by the vertical semi-flow given by: \[\varphi_{f}: (\partial_+ \Mf \times (-\infty, \infty])  \times (-\infty, \infty] \to (\partial_+ \Mf \times (-\infty, \infty])\]
		\[((x,t),s) \mapsto(x,t+s)\]
		
		The analogous statement holds in a product neighborhood $\partial_- \Mf \times [-\infty, \infty)$ of $\partial_- \Mf$. 
	\end{proposition} 
	
	\begin{proof}
		Consider the quotient from the partially compactified space $\pi: \tilde{M}_\infty \to \Mf$. We work in the positive ``end'' of the manifold, but note that the proof is analogous for the negative ``end''. Choose a juncture $j_+$ which bounds a tight ladder $U_+$ in $S$. Consider the subsurface $R_+=\overline{U_+ - f(U_+)} \subset S$ bounded by $j_+$ and $f(j_+)$. Observe that $R_+\times (-\infty, \infty]$ is a $\hat \varphi$-invariant fundamental domain for the action of $F$ on $\mathcal{U}_+ \times (-\infty, \infty]$. Via \Cref{boundarysurfaces}, it descends to the quotient $\M_f$ as a product region $\partial_+ \M_f\times (-\infty, \infty]$ which inherits the vertical flow. 
	\end{proof}
	
	Keeping in mind the product structure of $\eL$ in the boundary, a \textit{spiraling neighborhood} of a component $\Sigma \subset \partial \Mf$ is a closed collar neighborhood $U_{ \Sigma}$ of $\Sigma$ in $\Mf$ so that: 
	
	\begin{enumeri}
		\item $U_{ \Sigma}$ is a product foliated by arcs of $\eL$, and 
		\item $\partial U_{ \Sigma} = \Sigma \sqcup \Sigma_U$, where $\Sigma_U \cong \Sigma$ is transverse to $\eL$ and the foliation $\mathcal{F}$. 
	\end{enumeri}
	
	We refer to a spiraling neighborhood of $\partial \Mf$ (or $\partial_+ \Mf$ or $\partial_- \Mf$) as a union of disjoint spiraling neighborhoods over each of its components (see \cite[Section 3.1]{LMT23}).
	
	\begin{lemma}\label{spiralingNbhds}
		Given a positive juncture $j_+$ = $\partial U_+$ there exists a spiraling neighborhood $U^+$ of $\partial_+ \Mf$ so that $\overline{U_+} = S\cap U^+$.
	\end{lemma}
	Note that the analogous statement holds for a negative juncture.
	
	\begin{proof} 
	Consider the surface $R_0 = (\overline{U_+ - f(U_+)})\times \{0\} \subset S\times \R$. Use the product flow to carry $R_0$ to a surface $\tilde{R_0}$ whose interior is transverse to the flow and bounded by the multiloops $j_+\times\{0\}$ and $f(j_+)\times \{-1\}$. Observe that under the quotient map, $R=\pi(R_0)$ is a closed surface homeomorphic to $\partial_+ M$ and $M \bbslash R$ has a product component foliated by $\eL$ and bounded by $\partial_+ M$ and $R$. We isotope $R$ in a neighborhood of $j_+$ to be transverse to the flow while retaining these properties so that it meets $S_0$ only at $j_+\times \{0\}$ and bounds a spiraling neighborhood $U^+=\{\varphi_f(x, t)\mid x\in R,\, t\geq 0\}$ containing $\{\varphi_f(x, k)\mid x \in R_0,\, k\in \Z_{\geq 0}\} = \overline{U_+}$. Similarly, we see that
	\[ S_0 \cap U^+ = \{\varphi_f(x, k)\mid x\in S_0\cap R_0,\, k\in \Z_{\geq 0}\} = \overline{U_+}.\qedhere\]\end{proof}

	\subsubsection{Junctures at infinity}\label{subsec:junctures-at-infty}
		We say a positive or negative juncture $j$ satisfies the \textit{juncture distinction property} if for all $k\in \Z_{\neq 0}$ we have $f^k(j) \cap j = \emptyset$.  We note that for a \textit{pure} end-periodic mapping class, i.e. one which acts by the identity on $\Ends(S)$, each juncture satisfies the juncture distinction property. This is not necessarily true when an end has period $\geq2$.
		Note that this is an adjustment of \cite{CCF21}'s notion of the juncture intersection property---their definition requires any pair of components of $\{f^k (j)\}$ to be either disjoint or coincide, while our definition disallows the latter.  We emphasize that we consider junctures as loops, i.e. subsets, and encourage the reader to temporarily refrain from identifying a loop with its isotopy class.
		
		 Given a juncture $j \subset S_0$ which satisfies the juncture distinction property, we define its \textit{juncture annulus} $A_{j}$ as the following union of $\pi_1$-injective, embedded annuli. For a loop $\gamma \subset \mathcal{U}_+ \cap \mathcal{U_-}$, we define 
		 \[A_{\gamma} := \{\varphi_f(x, t) \,\mid x\in \gamma,\,t \in [-\infty,\infty]\}.\] 
		  If $\gamma$ is only positively escaping, then its associated annulus $A_{\gamma}:=\{\varphi_f(x, t) \,\mid x\in \gamma,\,t \in (-\infty, \infty]\}$; while if it is only negatively escaping then 
		$A_{\gamma}:=\{\varphi_f(x, t) \,\mid x\in \gamma,\,t \in [-\infty, \infty) \}.$
		 We let $j_\infty$ denote $A_j \,\cap\, \partial M$, the boundary of the juncture annulus ``at'' infinity. For a core $C$, we let $\partial_\infty C$ denote the multiloop of $\partial M$ formed by the union \[\partial_\infty C = (\partial_+C)_{\infty} \cup (\partial_-C)_{\infty}.\]
		Using \Cref{product}, we can produce a juncture annulus and juncture orbit from particular loops of $\Sigma$.
		
		\begin{observation}\label{annuloops}
			Suppose that $j_\infty \subset \partial M$ is the boundary of a juncture annulus. Let $\gamma$ be a simple closed loop of $\Sigma -j_\infty \subset \partial M$. The flow lines terminating in $\gamma$ form a $\pi_1$-injective annulus ${A}_\gamma \cong \gamma \times (-\infty, \infty]$ which meets $S_0$ in a multiloop $J_\gamma$ which is the orbit of an escaping loop of $f$.
		\end{observation}

	\subsubsection{Properly embedded cores}\label{propEmbedCore}

		\begin{corollary}\label{coresToGuts}
			Suppose that $C_0$ is a core whose junctures $\partial_+ C_0$ and $\partial_- C_0$ both satisfy the juncture distinction property. Then there is a isotopy along $\eL$ which carries $C_0$ to a properly embedded surface $(\bar C_0,\partial \bar{C_0}) \hookrightarrow (M,\partial M)$ which is transverse to $\eL$.
		\end{corollary}
		
		\begin{proof}
			Because the boundary of $C_0$ satisfies the juncture distinction property, the compact cylinders $A^+= {\partial_+ C} \times [0, \infty]$ and $A^-= {\partial_- C}\times [-\infty, 0]$ are each disjoint from one another and intersect $C_0$ only at its boundary.  Consider the homotopy equivalence of $M$ obtained by collapsing the cylinders $A^+$ by the constant map $(x,t) \mapsto (x, \infty)$ and $A^-$ by the map $(x,t) \mapsto (x, -\infty)$. The image of $C_0$ under this homotopy equivalence is a properly embedded surface and since it was obtained by collapsing flow lines of $\varphi_f$, the properly embedded surface remains transverse to $\eL$. 
			
			Observe that if the junctures of a core do not satisfy the juncture distinction property, then this process of collapsing juncture annuli will not yield an embedding of the surface with boundary.
		\end{proof}
		
		More generally, we say a properly embedded surface $C$ is a \textit{properly embedded core}, or more specifically an \textit{$\eL$-embedding} of a core $C_0$ if it is the image of a map $\iota: (C_0, \partial C_0) \to (\Mf, \partial \Mf)$ which is isotopic to the inclusion via an isotopy along the flow lines of $\eL$. The following condition ensures that the $\eL$-embedding of the core is a \textit{prefiber} of $\eL$, i.e. a properly embedded surface which satisfies  $\ell \cap C_0 \neq \emptyset$ for each flow line $\ell$ of $\eL$. 
		\begin{proposition}\label{coreFiber}
			A core $C_0\subset S_0$ meets each flow line of $\eL$ if and only if the orbit $ \{f^k(C_0)\}_{k\in\Z}$ covers $S_0$. 
		\end{proposition}
		\begin{proof}
			Here we identify $S$ with the fiber $S_0=S\times \{0\} \subset M_f$. Let $x \in S_0$ and let $\ell_x$ denote the flow line of $\eL$ which meets $x$. Note that $\ell_x \cap S_0 = \{ f^k(x)\mid k\in\Z\}$ and $\ell_x = \ell_{f^k(x)}$ for all $x\in S_0$ and $k\in\Z$. If $\ell_x$ meets $C_0$ for all $x\in S$, there exists a power $k$ such that $f^k(x) \in C_0$. Thus, $S \subset \bigcup_{k\in\Z} f^k(C_0)$. Similarly, assume that $\bigcup_{k\in\Z} f^k(C_0) = S_0$, i.e. for each $x\in S_0$, there exists a power $k$ so that $f^k(x) \in C_0$. Then $\ell_{f^k(x)} \cap C_0 = \ell_x \cap C_0 \neq \emptyset$, meaning each flow line meets $C_0$. 
		\end{proof}

\subsection{Geometrization of the mapping torus}
	 For $M$, a compact irreducible $3$-manifold with incompressible boundary, we say a compact surface $Y$ is \textit{essential} in $M$ if it is $\pi_1$-injective, properly embedded and the inclusion $i: (Y,\partial Y) \to (M, \partial M)$ is not isotopic (rel. $\partial M$) into $\partial M$.  $M$ is \textit{acylindrical} if it does not admit an essential annulus and it is \textit{atoroidal} if it does not admit an essential torus. 
	
	\begin{theorem}[Field--Kim--Leininger--Loving {\cite[Proposition 3.1]{FKLL23}}]\label{fkllhyp}
		If $f$ is an atoroidal\footnote{Their theorem is stated in terms of an irreducible end-periodic homeomorphisms, but their proof covers the more general atoroidal setting.} end-periodic homeomorphism, then $\Mf$ is a compact irreducible atoroidal $3$-manifold with incompressible boundary. Further, if  $f$ is strongly irreducible, then $\Mf$ is acylindrical.
	\end{theorem}
	
	We briefly discuss the geometrization of $M_f$. Denote by $\AH(M_f)$ the space of hyperbolic structures on $M_f$, considered as the space of discrete, faithful representations $\rho: \pi_1(M_f) \to \Isom(\H^3)$ up to conjugation.  It follows from \Cref{fkllhyp} and Thurston's Geometrization of Haken Manifolds \cite{Th86} (c.f. \cite[Theorem 1.43]{Kap01}) that $M_f$ admits a hyperbolic structure. Any geometrically finite structure $s \in AH(M_f)$ yields a convex-hyperbolic structure on $\Mf$ as its compact convex core. As discussed in the introduction, when $\Mf$ is acylindrical we let $s^*_f \in \AH(M_f)$ denote the unique metric whose convex core has a totally geodesic boundary \cite[Section 4]{Mc90}.
	
	\begin{corollary}
		If $f$ is an atoroidal end-periodic homeomorphism, then $M_f$ admits a non-empty deformation space $\AH(M_f)$ of hyperbolic structures. When $f$ is strongly irreducible, there exists a unique metric $s^*_f\in \AH(M_f)$ whose convex core admits a totally geodesic boundary structure.
	\end{corollary}
	
	 By the Ahlfors--Bers isomorphism theorem (see \cite[Section 8.7]{Kap01} for reference), the interior of $\AH(M_f)$, represented by the geometrically finite hyperbolic structures on $M_f$, is parameterized by the Teichm\"uller space  $\Teich(\partial \Mf)$. 
	
	\section{Setup and plan of the paper}\label{outline}
	\subsection{Setup}
		Here we fix notation for use in Sections 3-7. Let $f$ be a fixed atoroidal end-periodic homeomorphism of $S$, a boundaryless surface of infinite type with finitely many ends, all non-planar. We assume that $f$ is not a translation. We fix a hyperbolic structure $X$ on $S$ and let $\Lambda^+$, $\Lambda^-$ denote the positive and negative Handel--Miller laminations of $f$. We let $\Cmin$ denote a fixed choice of minimal core of $f$. A subsurface $Y \subset S$ will always denote an essential subsurface which meets both $\Lambda^+$ and $\Lambda^-$. We bundle the capacity $(\chi(f), \xi(f))$ into a quantity $\capgen(f)= \xi(f)^2 + \chi(f) +1$ called the \textit{capacity genus} of $f$. 
		
		Let $M_f$ be the (boundaryless) mapping torus of $f$ with $M=\Mf$, its compactification as a $3$-manifold with boundary as in \Cref{cpctification}. We consider $S=S_0$ as a fixed leaf of the depth-one foliation $\mathcal{F}$ defined in the previous section. We let $\eL$ denote the $1$-dimensional foliation of $M$ by flow lines of the suspension semi-flow $\varphi_f$.
		
		For $s\in\AH(M_f)$, we let $(M_f, s)$ denote the corresponding hyperbolic manifold with its induced convex core $(M, s)$. We consider only geometrically finite structures in this work. When $f$ is strongly irreducible, $(M_f, s^*)$ denotes the unique structure whose convex core has totally geodesic boundary. For a multicurve $\alpha$ in $S$ and $s \in \AH(M_f)$, we let $\ell_s(\alpha)$ denote the total length of its unique geodesic representative $\alpha^*$ in $(M_f, s)$.
		
	\subsection{Overview}
		Throughout the paper we will examine $M_f$ from various perspectives. We are interested in $M_f$ as a fibered manifold and although any fiber of $M_f$ must be of infinite type, we uncover the geometric structure of the manifold by appealing to the finite-type setting.
		Our primary tools come from recent work of Landry--Minsky--Taylor \cite{LMT23} who embed $M$ into a closed fibered hyperbolic manifold via their  \emph{$h$-double} construction. They produce a closed manifold $N=N(M,h)$ by identifying two copies of the compact manifold with boundary $M$ via a component-wise homeomorphism $h: \partial M \to \partial M$. The authors show that for the right choice of $h$, the resulting $h$-double $N=N(M,h)$ is fibered and hyperbolic. This allows tools for closed  fibered hyperbolic manifolds to be applied to $M_f$. \footnote{We note to the reader that our notation for $N$ and $M$ is swapped from that of Landry--Minsky--Taylor.} 
		
		In the language of Landry--Minsky--Taylor, the  $h$-double construction provides a \emph{spun pseudo-Anosov package}  $(N, \mathcal{F}, \varphi^\flat, M)$. This package consists of a closed fibered hyperbolic $3$-manifold $N$; a depth-one foliation $\mathcal{F}$ of $N$ which is transverse to $\varphi^\flat$,  a (dynamic blow up of) a circular pseudo-Anosov flow on $N$; and a compact manifold with boundary $M$ obtained from $N$ by cutting along the depth-zero leaves of $\mathcal{F}$.
		
		Explicitly, $S$ lives as a leaf of $\mathcal{F}$ and the \textit{spun pseudo-Anosov} representative $f^\flat$, determined by the spun pseudo-Anosov package, is the first return map of $S$ under the circular flow $\varphi^\flat$ which is isotopic to the original end-periodic homeomorphism. 
		
		A significant chunk of this paper is devoted to proving the following lengthy proposition. 
		In brief, this proposition facilitates a careful application of Minsky's \Cref{minsky} to $N$ so that we may coherently translate from the pseudo-Anosov lamination data and the geometry of $N$ to the end-periodic Handel--Miller lamination data and geometry of $M$. 
		
		\begin{proposition}[A core and gluing package]\label{spf}
			Let $C_0$ be a core of $f$ and $h:\partial M \to \partial M$ be a gluing homeomorphism so that: 
			\begin{enumerate}[label=(\arabic*)]
				\item $C_0$ is an $f$-extended core which meets each flow line of $\eL$.
				\item $\chi(C_0) +1  \leq \capgen(f)$. 
				\item $C_0$ admits an $\eL$-embedding $(C_0, \partial C_0)\hookrightarrow (M, \partial M)$; let $C$ denote its image.
				\item $h$ preserves $\partial C$ with $h_*([\gamma])=-[\gamma]\in H_1(\partial M)$ for each component $\gamma$ of $\partial C$. 
				\item $N=N(M,h)$ is atoroidal.
			\end{enumerate}
			
			Then:
			\begin{enumeri}
				\item\label{spf:Ngenus} The closed surface $\hat F=\hat F(C,h)$ obtained as the $h$-double of $C$ is isotopic to a fiber $F$ of $N$ of genus $g(F) \leq \capgen(f)$.
				
				\item\label{spf:lamicom}  Let $\lambda^+, \lambda^-$ denote the unstable and stable laminations of the pseudo-Anosov monodromy on $F$. Then, given a subsurface $Y$ with $d_Y(\Lambda^+, \Lambda^-) \geq 4$ there exists a subsurface $W \subset F$,  isotopic to $Y$ in $M$, so that 
				\[d_Y(\Lambda^+, \Lambda^-) \leq d_W(\lambda^+, \lambda^-) + 2.\]
			\end{enumeri} 
		\end{proposition}
		
		Here is a heuristic of the proof of \Cref{thma} assuming \Cref{spf} and provided a core $C$ and gluing homeomorphism $h:\partial C \to \partial C$ which satisfy its hypotheses.  We use \Cref{spf}\ref{spf:Ngenus} to ensure that the topology of the fiber of $N$ is controlled by $\capgen(f)$. 
		With \Cref{spf}\ref{spf:lamicom}, we may directly compare the subsurface projection data $d_Y(\Lambda^+,\Lambda^-)$ and $d_W(\lambda^+,\lambda^-)$ between a subsurface $Y\subset S$ and a subsurface $W\subset F$ isotopic to $Y$ in $N$ so that $\ell_N(\partial Y) = \ell_N(\partial W)$. We then apply Minsky's \Cref{minsky} to $N$ and conclude the argument by noting there exists a hyperbolic structure $s \in \AH(M_f)$ so that $(M_f, s) \hookrightarrow N$ is an isometric immersion. This ensures $\ell_s(\partial Y) = \ell_N(\partial Y)$, and the proof of the theorem follows. 
		
	\subsection{Plan of the paper}
		In Sections \ref{constructN} and \ref{subs} we record the proof of \Cref{spf} \ref{spf:Ngenus} and \ref{spf:lamicom}, assuming we have an abstract core $C$ and gluing involution which satisfies the hypotheses of the proposition. We use tools of Landry--Minsky--Taylor throughout. 
		 The bulk of the work is carried out in \Cref{whichCore} where we make an explicit choice of core $C$ and gluing homeomorphism $h$ which satisfy the hypotheses of \Cref{spf}. 
		 For certain strongly irreducible end-periodic homemorphisms, as when $\partial M$ consists of genus-$2$ surfaces in \Cref{thmb} (see also \Cref{pinchint} and \Cref{pinch}), we show that there is a $h$ for which $(\Mf, s^*_f)$ isometrically embeds into $N(M,h)$. \Cref{conj:tgmetric} holds for such homeomorphisms. 
		 
		We end the paper with \Cref{examples} which describes a few examples of families of end-periodic and strongly irreducible end-periodic homeomorphisms, both to serve as a reference to the reader and to prove \Cref{thmc} via an application of \Cref{thmb}.
		We provide plenty of examples to which one may apply \Cref{thmb} to obtain short curves in totally geodesic structures. Namely, for a fixed core $C$ of a handle shift $\rho$ on the ladder surface, we construct families of strongly irreducible end-periodic homeomorphisms $\{f_n\}$ with uniformly bounded capacity so that 
		\[\lim_{n\to\infty} d_C(\Lambda_n^+,\Lambda_n^-) \to \infty,\]
		hence $\ell_{s^*_n}(\partial C) \to 0$\ in the respective metrics with totally geodesic boundary. 
		
	\section{The doubling construction}\label{constructN}
	We begin this section by describing the $h$-double construction of Landry--Minsky--Taylor which produces a closed fibered 3-manifold $N$ in which $M$ embeds. The homotopy equivalence of surfaces obtained from \Cref{propEmbedCore} is described by \cite[Section 3]{LMT23} in terms of spiraling neighborhoods and \textit{juncture classes} of $\partial M$. It will be helpful to compare their terminology with ours. The fibration $M_f \to S^1$ induces a cohomology class $[j] \in H^1(\partial M)$, known as the \textit{juncture class} of $f$; this is the pullback of the fibration class $H_1(M_f) \to \Z$ under the inclusion map $\partial M \hookrightarrow M$. Indeed, for any properly embedded core $C\subset S$, the first cohomology class $[\partial C] \in H^1 (\partial M)$ corresponding to the algebraic intersection form is precisely the juncture class $[j]$ of $f$.
	
	Let the tuple $(M, \mathcal{F}, \eL)$ denote $M$, the oriented $3$-manifold with boundary;  $\mathcal{F}$, the co-oriented $2$-dimensional foliation; and $\eL$, the oriented $1$-dimensional foliation induced by the suspension flow of $f$. Let $(M^\downarrow, \mathcal{F}^\downarrow, \eL^\downarrow)$ denote the same with the opposite (co-)orientations in each factor. Let $\sigma: M \to M^\downarrow$ be the identity map, which, by definition, reverses the (co-)orientations of $M$, $\mathcal{F}$ and $\eL$.

	Given a component-wise orientation-preserving homeomorphism $h: \partial M \to \partial M$, the \textit{$h$-double of $M$} is the closed manifold defined by \[N(M,h) = M \cup_{\sigma|_{\partial M} \circ h} M^\downarrow,\] i.e. by gluing $\partial M$ and $\partial M ^\downarrow$ by the composition $\sigma|_{\partial M} \circ h $. We use $M \cup_h M^{\downarrow}$ as shorthand for this gluing or drop the $h$ decoration entirely. Note that $N(M,h)$ inherits its orientation from those of $M$ and $M^\downarrow$. The manifold $N$ admits a $1$-dimensional foliation obtained from the union $\eL \cup \eL^\downarrow$ and depth-one foliation obtained by $\mathcal{F} \cup \mathcal{F}^\downarrow$.  We denote the extensions of $\eL$ and $\mathcal{F}$ by the same notation in the context of the closed manifold. Suppose $W$ is a properly embedded surface $(W, \partial W) \hookrightarrow (M,\partial M)$, and $h(\partial W) = \partial W$. Then for $W^\downarrow = \sigma(W)$, the union $W \cup W^\downarrow$ is a closed surface of $N(M,h)$ which we will denote by $\hat F(W,h)$.
	
	\begin{figure}
		\centering
		\includegraphics[width=.5\textwidth]{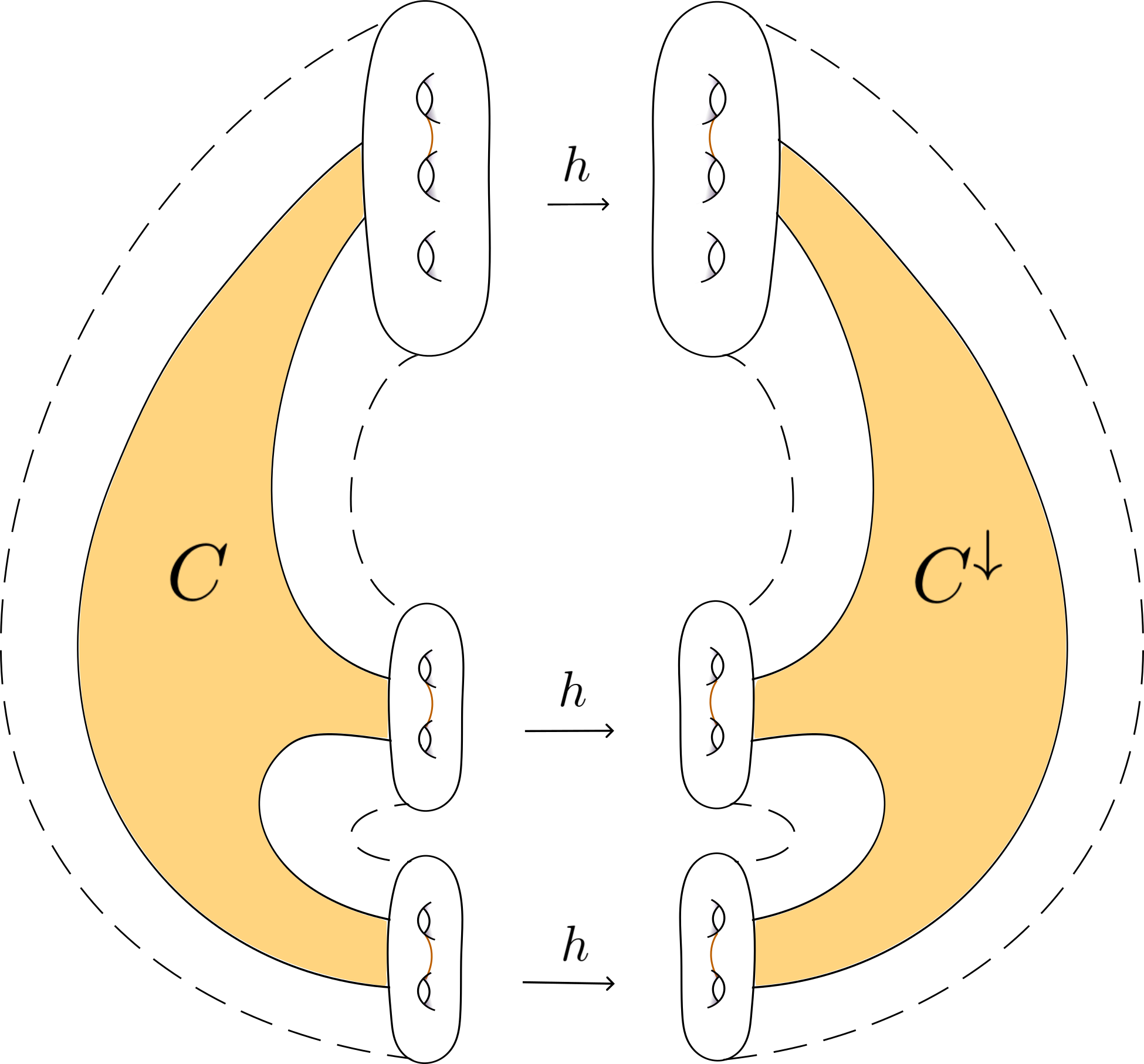}
		\caption{A cartoon of $N =N(M,h)$, the $h$-double of $M$. The closed surface $\hat F = C \cup_h C^\downarrow$ is formed by the double of a properly embedded core $C$.}
	\end{figure}
	
	 We rephrase a proposition \cite[Proposition 3.9]{LMT23} of Landry--Minsky--Taylor below. Their result is stated in more generality; they only require that $h$ acts by $-1$ on the juncture class and do not require $\partial C$ to be preserved under $h$.
	
	\begin{proposition}[Landry--Minsky--Taylor \cite{LMT23}]\label{lmtfiber}
		Let $C$ be a properly embedded core which meets each flow line of $\eL$. Let $h: \partial M \to \partial M$ be a component-wise homeomorphism satisfying $h(\partial C)=\partial C$ and $h_*([\gamma])=-[\gamma]$ for each component $\gamma \subset \partial C$.
	\end{proposition}
	\Cref{spf}\ref{spf:Ngenus} follows directly from \Cref{lmtfiber} whose argument we summarize below.	
	\begin{proof}[Proof of \Cref{spf}\ref{spf:Ngenus}]
	 	Suppose we have a core $C_0$ which meets each flow line and that $C$ is the image of an $\eL$-embedding of $C_0$. Further, assume $h$ preserves $\partial C$ with $h_*([\gamma])=-[\gamma]$ for each component $\gamma \subset \partial C$, and that $N=N(M,h)$ is atoroidal. We wish to show that the doubled surface is isotopic to a cross section of $\eL$, i.e. a fiber of $N$.  
	 
	 	As properly embedded cores, the interiors of the surfaces $C$ and $C^\downarrow$ are positively transverse to $\eL$ and $\eL^\downarrow$ and meet each flow line of the respective flows.  The boundaries of the co-oriented surfaces $C \subset M$ and $C^\downarrow \subset M^\downarrow$ are identified under the $h$-double gluing to form a new closed surface $\hat F= \hat F(C,h)= C \,\cup\, C^\downarrow$. We observe that in the closed manifold $N$, the closed surface $\hat F$ meets each flow line of the identified flow $\eL$. 
	
		We need to verify that $\hat F$ is consistently positively transverse to this flow at the source of the gluing. It suffices to show that the induced co-orientations of $C$ and $C^\downarrow$ are consistent in the doubled manifold. Since $C$ is positively transverse to $\eL$, it induces a co-orientation of $\partial C \subset \partial M$, represented by $[\partial C] \in H^1 (\partial M)$. Likewise, the positive co-orientation of $C^\downarrow$ induces a co-orientation of $\partial C ^\downarrow \subset \partial M^\downarrow$, represented by $[\partial C^\downarrow] \in H^1 (\partial M^\downarrow)$.  We check that this induces a consistent co-orientation of $\partial C \subset \partial M \subset N$. Indeed, since $\sigma^*[\partial C] = -[\partial C^\downarrow]$, it follows that $\sigma^*h^*([\partial C]) = \sigma(-[\partial C]) = [\partial C^\downarrow]$. 
	
		Now, since $\partial C$ is contained in a product neighborhood of $\partial M$ foliated by product fibers of $\eL$, we may isotope $\hat F$ along flow lines to a closed surface $F$ which is transverse to $\eL$ at each point. Thus, $F$ meets each flow line of $\eL$ and the first return map represents the monodromy of the fibration of the atoroidal manifold. 
		
		The last claim follows from a simple calculation: for a compact surface $C$ of genus $g$ with $b$ boundary components, its double has genus $|2g+b-1| \leq |\chi(C)| + 1\leq \capgen(f)$.
	\end{proof}
	
	The following remark follows from the construction of Landry--Minsky--Taylor \cite[Section 3]{LMT23}, and Thurston's Hyperbolization Theorem for fibered manifolds \cite[Theorem 0.1]{Th86II}. 
	
	\begin{remark}\label{ourspf}
		Any $C$ and $h$ as in \Cref{spf} produces a spun pseudo-Anosov package $(N, \mathcal{F}, \varphi^\flat, M)$ for $f$. 
		\begin{enumeri}
			\item $N=N(M,h)$ is a fibered hyperbolic manifold with fiber $F$ isotopic to $\hat F=\hat F(C,h)$. The first return map of $\eL$ for the fiber $F$ is isotopic to a pseudo-Anosov homeomorphism $\phi$.
		
			\item The pseudo-Anosov homeomorphism $\phi: F \to F$ induces a pseudo-Anosov suspension flow $\varphi$ on $N$. $S$ is isotopic to a surface $S'$ which is transverse to a dynamic blow up $\varphi^\flat$ of $\varphi$. 
			\item Let $f^\flat$ denote the first return map of $S'$ under $\varphi^\flat$; this map is called a spun pseudo-Anosov representative of $f$ and is isotopic to $f$ (via conjugation of the isotopy above). 
		\end{enumeri}
	\end{remark}

\section{Junctures, laminations and subsurfaces}\label{subs}
	
	As a reminder to the reader, a subsurface $Y \subset S$ will always denote a compact connected essential subsurface which meets both $\Lambda^+$ and $\Lambda^-$--we say such a subsurface is \textit{suitable} for $f$. By the (pseudo) $f$-invariance of the Handel--Miller laminations from \Cref{f-invariance}, $\Lambda^\pm$  meets $Y$ if and only if $\Lambda^\pm$ meets $f^n(Y)$ for all $n\in\Z$, and we have $d_Y(\Lambda^+,\Lambda^-) = d_{f^n(Y)}(\Lambda^+, \Lambda^-)$.
	
	We will apply the following remark to $\mathcal{J}_\pm$ or one of its partial juncture orbits $\mathcal{J}_\pm^k$ so that the closure of the union of geodesics is the negative (resp. positive) Handel--Miller lamination.
	
	\begin{lemma}\label{prop:lamConvergence}
		Let $\mathcal{G}$ be an union of complete geodesics of a hyperbolic surface $X$ and $\Lambda=\bar{\mathcal{G}}$. If $\Lambda$ intersects a geodesic $\gamma$ transversely, then so must $\mathcal{G}$. 
	\end{lemma}
	
	\begin{proof}
		Let $\tilde \Lambda$, $\tilde{\mathcal{G}}$, $\tilde \gamma$ be the respective lifts of $\Lambda$, $\mathcal{G}$ and $\gamma$ to the universal cover $\tilde X = \H^2$. The endpoints of each of the lifts define subspaces $\partial^2(\Lambda)$, $\partial^2(\mathcal{G})$ in the double boundary $\partial^2(\H^2)$. 
		Let $\tilde l$ $\subset$ $\tilde \Lambda$ be a leaf which meets a component $\tilde \gamma_0 \subset \tilde \gamma$. Since $\tilde l$ represents a limit point of $\partial^2 (\mathcal{G})$, there exists a leaf $g \subset \mathcal{G}$ so that  $\tilde g \subset U \times V$, the neighborhood formed by the components $U$ and $V$ of $\partial \H^2 -\{\gamma_0\}$. Thus, $g$ and $\gamma$ must intersect.
	\end{proof}
	
	\begin{lemma}\label{lemma:comparelam}
		Suppose that $Y\subset S$ is suitable for $f$. Then for any pair of positive and negative juncture orbits $J_+$, $J_-$, we have:
		\begin{enumeri}
			\item $J_+$ and $J_-$ meet $Y$.
			
			\item $d_Y(J_\pm, \Lambda^\mp) \leq 1$
			\item $|d_Y(J_+, J_-)- d_Y(\Lambda^+,\Lambda^-)| \leq 2$
		\end{enumeri} 
	\end{lemma}
	It is worth noting that a subsurface which meets $J_\mp$ need not meet its limiting lamination $\Lambda^\pm$.
	
	\begin{proof}
		By \Cref{piY-HM}, if $\Lambda^+$ and $\Lambda^-$ meet $Y$, they must also meet $\partial Y$. It follows by applying \Cref{prop:lamConvergence} to $\mathcal{G}=\mathcal{J}_\pm$, $\bar{\mathcal{G}} = \Lambda^\mp$ and $\gamma = \partial Y$ that $\partial Y$ must also meet the juncture orbits $J_+$ and $J_-$. Recall that the projection map $\pi_Y$ is defined by lifting arcs and curves to the cover $\tilde X_Y$ corresponding to $Y$. By \Cref{disjoint:lamjunc}, which states that $\mathcal{J}_\pm$ and $\Lambda^\mp$ are disjoint in $X$, the corresponding lifts $\tilde{\mathcal{J}_\pm}$ and $\tilde{\Lambda^\mp}$ are disjoint in $\tilde X_Y$ so the respective projections satisfy $\pi_Y(J_\pm) \subset \mathcal{N}_1(\pi_Y(\Lambda^\mp))$ and $\pi_Y(\Lam^\mp) \subset \mathcal{N}_1(\pi_Y(J_\pm))$  in $\mathcal{AC}(Y)$. 
		Finally, for any $j_\pm \in J_\pm$ with $d_Y(j_+, j_-)= d_Y(J_+, J_-)$, and any $l_\pm \in \pi_Y(\Lambda^\pm)$ with $d_Y(l_+, l_-) =d_Y(\Lambda^+, \Lambda^-)$, it follows that 
		$$ d_Y(l_+, l_-) -2 \leq  d_Y(j_-, j_+) \leq  d_Y(l_+, l_-) +2. $$
	\end{proof}
	
	\begin{lemma}\label{lemma:junctureminl}
		Let $Y$ be suitable for $f$. Then for any juncture orbit $ J_+= \{f^k(j_+)\}_{k\in\Z}$ there exists a minimal $K \in \Z$ for which $\pi_Y(f^{K}(j_+)) = \emptyset$. 
	\end{lemma}
	One obtains the analogous statement, i.e. the existence of a maximal $K$ for negative juncture orbits by swapping $+$ for $-$ in the argument. 
	\begin{proof}
		Fix a juncture $j_+$ and let $\kappa =\{k_n\}$ be an enumeration of the non-empty set $\{k\in \Z \mid  \pi_Y(f^{k}(j_+)) = \emptyset \}$ (see \Cref{escaping}). 
		Pull $\partial Y$ and ${J}^\kappa = \bigcup f^{k_n}(j_+)$ tight to geodesic multicurves $\partial Y^*$ and $\mathcal{J}^\kappa$ on $X$ so that $\partial Y^* \,\cap\, \mathcal{J}^\kappa = \emptyset$. 
		
		Assume by contradiction that $\kappa$ is an arbitrarily negative sequence. Observe that closure of the geodesic multicurve $\bar{\mathcal{J}^\kappa}$ contains $\Lambda^-$ as a sublamination. So, by \Cref{prop:lamConvergence}, since $\partial Y^*$ meets $\Lambda^-$ it must also meet $\mathcal{J}^{\kappa}$, a contradiction. Thus, the sequence $\{k_n\}$ is bounded below and has a minimal element $K$.
	\end{proof}	
	
	\begin{lemma}\label{lemma:corelemma}
		Let $C$ be an $f$-extended core of $S$. Then given any surface $Y\subset S$ with $d_Y(\Lambda^+,\Lambda^-) \geq 4$, there exists $n\in \Z$ so that $f^n (Y) \subset C$.
	\end{lemma}
	
	\begin{proof}
		 Consider  $J_+= \{f^{k}(\partial_+C)\}_{k\in \Z}$  and  $J_-= \{f^{k}(\partial_-C)\}_{k\in \Z}.$ 
		 Let $C'$ be a subcore of $C$ with $\partial_\pm C' \subset J_\pm$. Recall that each component of $\partial C'$ is essential in $C$. Let $Y$ be an essential subsurface of $S$ with $d_Y(\Lambda^+,\Lambda^-) \geq 4$.
		Let $K$ be the minimal power (from \Cref{lemma:junctureminl}) for which $f^K(\partial_+C)$ is disjoint from $Y$. Set $W =f^{-K}(Y)$. While $W$ is disjoint from $\partial_+C$, we observe that $W$ must meet $\partial_+ C'$ by the minimality of $K$ and the fact that $Y$ is suitable.
		
		Now $W$ is in a favorable position with respect to $\partial_+ C$, so we look to  $\partial_- C'$. Either $\pi_{W}(\partial_- C') = \emptyset$, or $\pi_{W}(\partial_- C') \neq \emptyset$.  In the former case, $W$ is a connected surface of $S -(\partial_+ C \sqcup \partial_-C')$ and not contained in any ladder (see \Cref{piY-HM}), so it must be contained in the core $C$. In the latter case, the disjoint multicurves $\partial_-C'$ {and} $\partial_ +C'$ both meet $W$ so that $d_{W}(\Lambda^+ ,\Lambda^-) \leq {3}$, a contradiction. Thus, $W=f^{-K}(Y)$ is a subsurface of $C$.
	\end{proof}

	\subsection{Compatibility of laminations}
		We now require tools to compare the Handel--Miller laminations of $f$ to the pseudo-Anosov laminations of the fiber $F$ of $N$. We give a brief exposition of the singular laminations $\Lambda^+_\flat$ and $\Lambda^-_\flat$ obtained from a spun pseudo-Anosov package $(N, \mathcal{F}, \varphi^\flat, M)$ and refer the reader to Landry--Minsky--Taylor \cite[Section 4]{LMT23} for specifics.
		 
		Let $C$, $F$ and $h$ be as in \Cref{ourspf}. Since $F$ is a cross section of the foliation $\eL$ on the atoroidal manifold $N$, the first return map $r:F\to F$ of $\eL$ is isotopic to a pseudo-Anosov homeomorphism $\phi$ on $F$ \cite{Th82}. The pseudo-Anosov homeomorphism induces a pseudo-Anosov suspension flow $\varphi$ on $N$. 
		Recall that $\phi$ is equipped with a pair of unstable and stable foliations $\mathcal{F}^+$ and $\mathcal{F}^-$ of $F$, and corresponding unstable and stable laminations $\lambda^+$, $\lambda^-$ of $\phi$ (see \Cref{laminatioNs}). The circular flow $\varphi^\flat$ is obtained from $\varphi$ by a \textit{dynamic blow up}, an operation on $\varphi$ which alters the flow at its singular orbits; we refer the reader to \cite[Section 2.2]{LMT23} for specifics on this blow-up procedure. Under this process, the fiber $F$ remains transverse to $\varphi^\flat$ and still meets each flow line of $\varphi^\flat$. More crucially, up to an isotopy of $N$, the depth-one foliation $\mathcal{F}$ of $N$ is positively transverse to $\varphi^\flat$. 
		Given this isotopy, we make no distinction between $S$ and the isotopic copy of $S$ which is transverse to the flow.
		
		The circular flow $\varphi^\flat$ induces unstable and stable $2$-dimensional singular foliations of $N$ which we denote by $W_\flat^+$ and $W_\flat^-$. By intersecting the leaves of $W_
		\flat^{+/-}$ with $S$, we obtain unstable and stable singular foliations $\mathcal{W}_\flat^+$ and $\mathcal{W}_\flat^-$ of $S$. So as to not overwhelm the reader, we use the same notation for the  foliations obtained by the intersection of $\mathcal{W}_\flat^{+/-}$ with $F$.
		Landry--Minsky--Taylor introduce the \textit{positive and negative invariant singular (sub)-laminations} of a spun pseudo-Anosov map, which we denote by $\Lambda_{\flat}^+$, $\Lambda_{\flat}^-$ for the spA representative $f^\flat$. In brief, the positive lamination is obtained by removing negatively escaping leaves and contracting half-leaves from the unstable foliation $\mathcal{W}_\flat^+$, and the negative lamination is obtained by removing positively escaping leaves and expanding half-leaves from the stable foliation $\mathcal{W}_\flat^-$ \cite[Proposition 5.8]{LMT23}. In fact, Landry--Minsky--Taylor provide a direct comparison between the Handel--Miller laminations $\Lambda^\pm$ and $\Lambda_{\flat}^\pm$ by lifting them to the universal cover $\tilde X$. Here, given a singular sublamination $\Lambda$ of a singular foliation of $X$, $\partial ^2 (\Lambda)\subset \partial^2(\H^2)$ denotes the end points of leaf lines of $\tilde \Lambda$ in the double boundary of $\H^2$.
		
		\begin{theorem}[Landry--Minsky--Taylor {\cite[Theorem 8.4]{LMT23}}]\label{HMcomparison}
			$$\partial^2\left({\Lambda^\pm_{\flat}}\right) = \partial^2\left({\Lambda^\pm} \right).$$ 
		\end{theorem}
		
		We are now ready to prove \Cref{spf}\ref{spf:lamicom}.
		
		\begin{proof}[Proof of \Cref{spf}\ref{spf:lamicom}.]
			Let $Y$ be a subsurface of $S$ with $d_Y(\Lambda^+,\Lambda^-) \geq 4$. We wish to show that there is a subsurface $W \subset F$ isotopic to $Y\subset S$ in $N$ so that the projection distances $d_W(\lambda^+, \lambda^-)$ and $d_Y(\Lambda^+, \Lambda^-)$ are comparable. Since $C$ (satisfying the hypotheses of \Cref{spf}) is $f$-extended,  \Cref{lemma:corelemma} asserts that $Y$ is isotopic in $M$ to the surface $f^{-K}(Y)$ which is contained in $C$. We refer to this surface as $Y$ instead, as $d_Y(\Lam^+, \Lam^-)=d_{f^n(Y)}(
			\Lam^+, \Lam^-)$ for all $n$. There is a composition of isotopies which first carries $Y$ along $\eL$ into the subsurface $\hat F$ and then to a subsurface $W \subset F$. 
			The product of these isotopies defines a (fixed) homeomorphism from $Y \subset C$ to its image $W\subset F$ which allows us to pass from curves and proper arcs of $Y$ to curves and proper arcs of $W$. It is with this homeomorphism that we identify $\mathcal{AC}(Y)$  with $\mathcal{AC}(W)$.
			
			For brevity, we just discuss the positive/unstable (singular sub)-laminations and foliations; by replacing $+$ with $-$, our arguments hold for the negative/stable versions of the laminations and foliations. The lifts  $\tilde{\Lambda^+}$ and $\tilde{\Lambda^+_{\flat}}$ to $\H^2$ are $\pi_1(X)$-invariant singular laminations whose endpoints of leaf lines agree by \Cref{HMcomparison}. Note that the leaf lines of the lifts of $\tilde{\Lambda^+}$ and  $\tilde{\Lambda^+_{\flat}}$ to the cover $\tilde{X}_Y$ also agree. It follows that $\pi_Y(\Lambda^+)=\pi_Y(\Lambda^+_{\flat})$. Recall that $\Lambda_{\flat}^+$ is obtained from $\mathcal{W}_\flat^+$ by removing leaves and half-leaves so that $\pi_Y(\Lambda^+_{\flat}) \subset \pi_Y(\mathcal{W}_\flat^+)$ and $d_Y(\mathcal{W}_\flat^+, \Lambda^+_\flat) \leq 1$.   
			
			Finally, we compare the projections $\pi_W(\mathcal{F}^+)$ and $\pi_W(\mathcal{W}_\flat^+)$ corresponding to the genuine circular pseudo-Anosov flow $\varphi$ and the blown-up flow $\varphi^\flat$, respectively. Since the flow $\varphi^\flat$ is obtained from $\varphi$ by introducing \textit{blown annuli}, i.e. mapping tori of finite trees at singular orbits of $\varphi$, the singular foliation $\mathcal{F}^+$ can be re-obtained from the singular foliation $\mathcal{W}_\flat^+$ by collapsing the finite trees which produce the blown annuli. This is a homotopy equivalence of the surface which takes leaf lines of $\mathcal{F}^+$ to leaf lines of $\mathcal{W}^+_{\flat}$ so that $\pi_W(\mathcal{F}^+)=\pi_W(\mathcal{W}^+_{\flat})$.
			
			Thus, we have:
			\begin{align*}
				d_Y(\Lambda^+, \Lambda^-) = d_Y(\Lambda_{\flat} ^+, \Lambda_{\flat}^-) &\leq d_Y(\mathcal{W}_{\flat}^+, \mathcal{W}_{\flat}^-) + 2 \\
				&= d_W(\mathcal{W}_{\flat}^+, \mathcal{W}_{\flat}^-) +2  =d_W(\lambda^+, \lambda^-) +2.
			\end{align*}
		\end{proof}
			
		This completes the proof of \Cref{spf}\ref{spf:lamicom} and the proof of \Cref{spf}.

\section{Construction of core and gluing map}\label{whichCore}
	Our proofs of Theorems \ref{thma} and \ref{thmb} refer to a specific representative of a ``good'' gluing homeomorphism $h$ and a ``nice'' core $C_0$ chosen with respect to our fixed atoroidal end-periodic homeomorphism $f$. We use this section to specify the gluing map and core and show that they satisfy the hypotheses of \Cref{spf}. The algorithm to decide on a core is organized by two cases for $f$:
	\begin{enumerate}[label=\textbf{Type \Roman*}:, leftmargin=4\parindent]
		\item $f$ is strongly irreducible.
		\item $f$ is atoroidal and not strongly irreducible.
	\end{enumerate}
	
	We briefly explain the process to designate the core and homeomorphism; the process is streamlined when the map $f$ is assumed to be pure and we encourage the reader to assume so on a first reading.  First, we designate a particular multicurve $\gamma_{\partial M}$ which will serve as the isotopy class of $\partial_\infty C_0$. When $f$ is pure, each component of $\partial C_0$ is a separating curve in $S$ which bounds a nesting neighborhood of $f$ so that each component of $\partial_\infty C_0$ lies on a distinct component of $\partial M$. If $f$ is not pure we cannot guarantee the latter property. In this case, we instead choose a multiloop $j\subset S$ so that $\partial_\infty j=\partial_\infty C_0$ consists of $p_{\Sigma}$ parallel components about $\gamma_{\Sigma}$ for each $\Sigma$, where $p_\Sigma$ denotes the period of the corresponding end orbit. Next, we choose $h_0: \partial M\to \partial M$ in the isotopy class of a hyperelliptic involution homeomorphism chosen to preserve $\partial_\infty C_0$. 
	Additionally, if $f$ is of Type III, we compose $h_0$ by partial pseudo-Anosov homeomorphisms of the components $\Sigma \subset \partial M$ which admit essential annuli. 
	We nudge the reader to remind them of the subtleties of dealing with multiloops as opposed to multicurves.

	\subsection{Nice cores}
		We now fix some notation and choices for use in the rest of this section. 
		Recall that $\Cmin= S - (U_+ \,\sqcup\, U_-)$. Let $R = R_+ \,\sqcup\, R_-$, where $R_+ = \overline{U_+ - f(U_+)}$ and $R_- = \overline{U_- - f^{-1}(U_-)}$. Each component $\Sigma \subset \partial M$ has an associated end orbit $\mathcal{O} = \epsilon_0, \ldots , \epsilon_{p-1}$ of period $p=p_{\mathcal{O}}=p_{\Sigma}$ labeled so that $f^{\pm n}(\epsilon_0) = \epsilon_n$ (depending on whether $\mathcal{O}$ is attracting or repelling). Let $\mathcal{U}_i$ denote the escaping set of $\epsilon_i$. We describe a fixed choice of curve $\gammaSig \subset \Sigma$ for each component $\Sigma$. 
		
		\begin{observation}\label{obs:noreducing}
			There is a separating curve $j_0$ of $\mathcal{U}_0$ which is contained in $R$ and whose homotopy class is not reducing under $f$. 
		\end{observation}
		\begin{proof}
			By \Cref{noEscape}, we need only choose a separating curve $\gamma_0$ of $\mathcal{U}_0$ contained in $R$ with $\pi_{Y}(j_0, \Lambda^+)\geq 2$ for some subsurface $Y$ (supposing $\epsilon_0$ is an attracting end). Take a separating curve $j'_0$ homologous to the boundary of $U_0$, the component of $U_+$ which is a neighborhood of $\epsilon_0$. We can replace $j'_0$ with one that meets $\Lambda^+$ by, for example, applying enough powers of a pseudo-Anosov map supported on the component $R_0\subset R$.
		\end{proof}

		For each component $\Sigma$, choose $j_0(\Sigma)$ to be a separating curve of $\mathcal{U}_0$ contained in $R_0$ as in the observation. We let $\gammaSig= (j_0(\Sigma))_\infty$, i.e., the boundary of the juncture annulus of $j_0(\Sigma)$ which meets $\Sigma$ (recall from \Cref{subsec:junctures-at-infty}). Doing so produces a single curve for each component of $\partial M$ and we let $\gamma_{\partial M}$ denote the union of $\gammaSig$ over all components of $\partial M$.
		
	As promised in the introduction of the section, we construct a core $C_0$ which is $f$-extended, contains $\Cmin$ and whose boundary consists of $p_\Sigma$-parallel components about $\gamma_\Sigma$ for each component $\Sigma$ of $\partial M$. When $f$ is pure, it suffices to let $C_0$ be the core bounded by $j_+$ and $j_-$, where $j_+= \bigcup f(j_0(\Sigma))$ taken over all positive $\Sigma \subset \partial_+ M$ and $j_- = \bigcup f^{-1}(j_0{(\Sigma)})$ taken over all $\Sigma \subset \partial_- M$. Otherwise, if $f$ is not pure, we work in one connected component $\Sigma$ and its corresponding end orbit $\mathcal{O}$ of order $p$ at a time. Choose 
	\[j_{ \Sigma} = j_0 \sqcup \ldots \sqcup j_{p-1},\] 
	a disjoint union of distinct parallel copies of $j_0=j_0(\Sigma)$ in $R$. We further impose that $U_i \supsetneq U_j$ for $i<j$, where $U_i$ denotes the nesting neighborhood of $\epsilon_0$ which $\gamma_i$ bounds.
	\begin{figure}
		\centering
		\includegraphics[width=.75\textwidth]{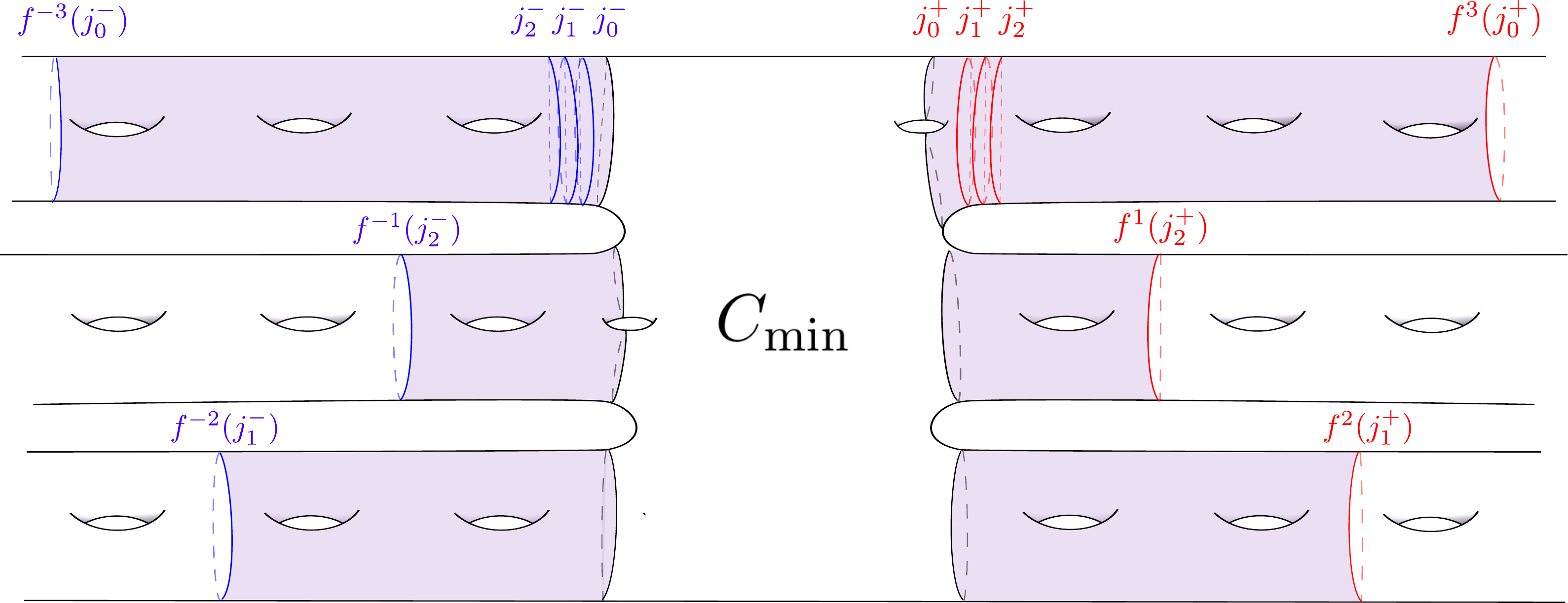}
		\caption{An $f$-extended core $C_0$ of an end-periodic map whose end orbits have period $3$. Here $j_{\mathcal{O}_-} = f^{-3}(j^-_0) \cup f^{-2}(j^-_1) \cup f^{-1}(j^-_2)$ and $j_{\mathcal{O}_+}=  f^3(j^+_0) \cup f^2(j^+_1) \cup f(j^+_2)$ so that $C_0$ is the union of $\Cmin$ and the disconnected surface $Z$ colored in purple. \label{gammaext}}
	\end{figure}
Let $j_{\mathcal{O}}$ denote the following multiloop, depending on whether the end orbit $\mathcal{O}$ is positive or negative.
	\[j_{\mathcal{O}} = \begin{cases}
		\displaystyle \bigcup_{i=0}^{p-1}  f^{p-i}(\gamma_i), & \text{if } \mathcal{O} \text{ is attracting. }\\
		\displaystyle \bigcup_{i=0}^{p-1}  f^{-(p-i)}(\gamma_i), & \text{if } \mathcal{O} \text{ is repelling. }
	\end{cases}	\]
	
	For each end orbit $\mathcal{O}$, the multiloop $j_\mathcal{O}$ bounds a tight (partial) ladder so that the union $\bigcup j_{\mathcal{O}}$ over all end orbits is the disjoint union of a positive and negative juncture of $f$. We let $C_0$ be the core bounded by these junctures. Observe that $C_0$ is $f$-extended and was carefully constructed to satisfy the juncture distinction property (see \Cref{subsec:junctures-at-infty}). Let $\Gamma=\partial_\infty C_0$ and $\Gamma_\Sigma = \Gamma \cap \Sigma$. 
	
	\subsection{Good homeomorphisms}
	
	To find a homeomorphism $h$ for which $N(M,h)$ is atoroidal, it suffices to verify that $h$ does not identify the boundaries of any two essential annuli of $M$. For each component $\Sigma \subset \partial M$, we let $\mathcal{A}_{\Sigma} \subset \mathcal{C}^{(0)}(\Sigma)$ denote the collection of curves which comprise of the boundaries of essential annuli of $M$ which meet $\Sigma$. Naturally, $\mathcal{A}_{\Sigma}$ is empty if $M$ is acylindrical.
	
	\begin{lemma}\label{hatoroidal} 
		Let $h$ be a component-wise homeomorphism of $\partial M$ for which $h(\mathcal{A}_{ \Sigma}) \cap \mathcal{A}_{ \Sigma} = \emptyset$ in $\mathcal{C}(\Sigma)$ for each component $\Sigma \subset \partial M$. Then $N(M,h)$ is atoroidal. 
	\end{lemma}
	\begin{proof}
		Suppose $T$ is an essential torus of $N(M,h)$. Since $M$ is atoroidal, the torus must meet $\partial M$. By a strengthening of a theorem of Roussarie--Thurston (see \cite[Theorem 2.9]{H86}), we may isotope $T$ so that the torus is transverse to $\mathcal{F}\subset N$. With a further isotopy we can remove any disks in $T - \partial M$ and assume $T$ meets $\partial M$ in an essential $1$-manifold $\Gamma = T \cap \partial M$. Namely $T \bbslash \Gamma$, the torus cut along the multicurve, is the disjoint union of compact essential annuli.
		
		If $f$ is strongly irreducible its mapping torus is acylindrical, thus, no such annulus exists and $N(M,h)$ is atoroidal.	Otherwise, let $\alpha$ be a component of $\Gamma$ and consider the annuli $A$ and $A^\downarrow \subset T \bbslash \Gamma$ whose boundaries meet at $\alpha$.  After cutting $N \bbslash \partial M$ we see curves $\alpha \subset M$, and $\alpha^\downarrow \subset M^\downarrow$, both bounding the respective annuli,  which were identified by $h(\alpha)=(\alpha^\downarrow)$. This contradicts the stated assumption. 
	\end{proof}
	
	The JSJ decomposition for manifolds with boundary provides $A$, a (possibly empty) disjoint collection of essential annuli, so that any essential annulus of $M$ is homotopic into a component of $M-A$ \cite[Theorem 3.8, 3.9]{Bon01}. In other words, we have either $\mathcal{A}_{ \Sigma} = \emptyset$ or $\diam_{\mathcal{C}(\Sigma)}(\mathcal{A}_{ \Sigma})\leq 2$.\\
	
	We say $h$ is a hyperelliptic involution homemorphism about an oriented loop $\gamma$ if it is order-$2$, preserves $\gamma$ setwise and acts by $-1$ on $H_1(\Sigma; \Z)$ (see Farb--Margalit \cite[Section 7.4]{FM12}).
	
		\begin{lemma}\label{h-multiloops}
			Suppose that $h$ is a hyperelliptic involution homeomorphism about an oriented loop $\gamma$.
			Then for each $p\geq 1$ there exists an oriented multiloop $\Gamma = \gamma_0 \cup \gamma_1 \cup \ldots \cup \gamma_{p-1}$ with each component isotopic to $\gamma$ so that $h(\Gamma)=\Gamma$ with $h_*([\gamma_i])=-[\gamma_i]$. 
		\end{lemma} 
		\begin{proof}
			We can always find sufficiently small $\varepsilon_1<\varepsilon_2$ and annular neighborhoods $\mathcal{N}_{\varepsilon_1}(\gamma)\subset \mathcal{N}_{\varepsilon_2}(\gamma)$ of $\gamma$ so that $h(\mathcal{N}_{\varepsilon_1}(\gamma)) \subset \mathcal{N}_{\varepsilon_2}(\gamma)$. Choose $\gamma_1, \ldots, \gamma_{p}$, a collection of distinct, disjoint, parallel copies of $\gamma$ within a component $\mathcal{N}_{\epsilon_1}(\gamma)\setminus \gamma$. Set $\gamma_{i+p}=h(\gamma_i)$ for all $1 \leq i \leq p$. The homeomorphism $h$ preserves $\gamma_1 \cup \ldots\cup \gamma_{2p}$ since it is order-$2$. The statement follows by restricting the number of curves and adding $\gamma=\gamma_0$ to the collection, if need be.  By orienting each $\gamma_i$ consistently we ensure that the co-orientation of $h(\gamma_i)$ is the negative of that of its preimage. 
		\end{proof}
		
		\begin{remark}\label{rem:h-multiloops}
			For each $\Sigma \subset \partial M$ there is a hyperelliptic involution homeomorphism $h_{\Sigma}$ about $\gamma_\Sigma$ which preserves the multiloop $\Gamma_{\Sigma}=\gamma^{\Sigma}_0 \cup \ldots \cup \gamma^{\Sigma}_{p-1} $ so that $h_*([\gamma^{\Sigma}_i])=-[\gamma^{\Sigma}_i]$. To see this, simply conjugate a hyperelliptic involution homeomorphism about a representative of $\gamma_\Sigma$ by a homeomorphism which takes $\Gamma_\Sigma$ to the relevant multiloop in \Cref{h-multiloops}. Let $h_0:\partial M \to \partial M$ denote the component-wise homeomorphism defined by $h_0|_{\Sigma}= h_\Sigma$.
		\end{remark}
		
		We are ready to describe the gluing homeomorphism $h: \partial M \to \partial M$ for each type. 
		
		\begin{casee}[A hyperelliptic involution about $\partial C$]\label{case-i}\normalfont
			Recall that $f$ is assumed to be strongly irreducible. We simply let $h=h_0$, the homeomorphism from \Cref{rem:h-multiloops} which preserves $\Gamma_\Sigma$ so that $h_*([\gamma^{\Sigma}_i])= -[\gamma^{\Sigma}_i]$. 
		\end{casee} 
		
		\begin{casee}
			[A homeomorphism which mixes annuli]\label{case-ii}\normalfont	
			Recall that in this type $f$ is atoroidal but not strongly irreducible. It follows that $M$ is not acylindrical, i.e. it admits an essential annulus. We consider $\Gamma$ as above and let
			 $W_{\Sigma}$ be a connected open annular neighborhood of $\Gamma_\Sigma$. We construct a homeomorphism $\phiSig$ which is supported on $\Sigma_0 = \Sigma - W_{\Sigma}$. If $\mathcal{A}= \mathcal{A}_{\Sigma} = \emptyset$, we let $\phiSig = \id_{ \Sigma}$. Otherwise, if $\mathcal{A}$ is non-empty we construct $\phi_\Sigma$ as follows. 
			
			Observe that for the homeomorphism $h_0$ above, both $\mathcal{A}$ and $h_0(\mathcal{A})$ meet $\Sigma_0$ since $\gammaSig$ is a connected non-separating curve of $\Sigma$. 
			Critically, ${\gammaSig} \notin \mathcal{A}$ since it was chosen as a curve which is not reducing (\Cref{obs:noreducing}). We will analyze $\mathcal{A}$, its image $h_0(\mathcal{A})$, and the projections $\pi_{\Sigma_0}(\mathcal{A})$ and $\pi_{\Sigma_0}(h_0(\mathcal{A}))$. 
			Let $D = d_{\Sigma_0}(\mathcal{A}, h_0(\mathcal{A}))$ and  
			choose a representative  $\phiSig$ $\in \Homeo(\Sigma_0, \partial\Sigma_0)$ of a pseudo-Anosov mapping class (i.e. one which is isotopic to a pseudo-Anosov homeomorphism after forgetting the boundary component) with $\tau_{{\Sigma_0}}(\phiSig) > D+3$. 
			We let $$h = \left(\prod_{\Sigma \subset \partial M} \phi_{\Sigma}\right) \circ h_0$$  
			Since $h|_{W_\Sigma} = h_0$, the homeomorphism $h$ preserves the prescribed multiloop $\Gamma$. For $\alpha \in\pi_{\Sigma_0}(\mathcal{A})$, we have:
			\begin{align*}
			d_{\Sigma_0}(h(\alpha), \alpha)&= d_{\Sigma_0}(\phiSig (h_0(\alpha)), \alpha) \\
			&\geq d_{\Sigma_0}(\phiSig (h_0(\alpha)), h_0(\alpha)) -  d_{\Sigma_0}(h_0(\alpha),\alpha)\\
			&\geq (D+3) -D =3. 
			\end{align*}
			As above, $\diam_{\mathcal{C}(\Sigma_0)}(\pi_{\Sigma_0}(\mathcal{A})) \leq 2$. So, it follows that for all $\alpha \in \mathcal{A}$, $h(\alpha) \notin \mathcal{A}$. Thus, $h$ satisfies the hypotheses of \Cref{hatoroidal}.
		\end{casee} 
		
		The next lemma shows that our choice of $C_0$ and $h$ satisfy the hypotheses of \Cref{spf}. 
		
		\begin{lemma}\label{goodCore}
			In \Cref{case-i} and \ref{case-ii} above, the core $C_0$ and homeomorphism $h$ satisfy the following.
			\begin{enumerate}[label=(\arabic*)]
				\item $C_0$ is an $f$-extended core which meets each flow line of $\eL$.
				\item $\chi(C_0) +1 \leq \capgen(f)$. 
				\item $C_0$ admits an $\eL$-embedding $(C_0, \partial C_0)\hookrightarrow (M, \partial M)$; let $C$ denote its image.
				\item $h$ preserves $\partial C$ with $h_*([\gamma])=-[\gamma]\in H_1(\partial M)$ for each component $\gamma$ of $\partial C$. 
				\item $N=N(M,h)$ is atoroidal.
			\end{enumerate}
		\end{lemma}	
		
		\begin{proof}[Proof of (1)]
			The core $C_0$ is constructed as an $f$-extension. It contains the core bounded by $f^{-1}(f^{p-1}(\gamma_1) \cup f^{p-2}(\gamma_2) \cup \ldots \cup f^p(\gamma_0))$ for the attracting orbits and  $f(f^{-(p-1)}(\gamma_1) \cup f^{-(p-2)}(\gamma_2) \cup \ldots \cup f^{-p}(\gamma_0))$ for the repelling orbits. We prove that it meets each flow line of $\eL$. Recall that $U_{i}$ denotes the neighborhood of $\epsilon_0$ which $\gamma_i$ bounds. For each end orbit $\mathcal{O}$ of period $p$, let $L_\mathcal{O}$ denote the compact connected subsurface:
			 
			 \begin{equation*}
			 L_{\mathcal{O}} = 
				 \begin{cases}
				 	\bar{U_{0}- f^p(U_{0})} & \text{if } \mathcal{O} \text{ is attracting, }\\
				 	\overline{U_{0}- f^{-p}(U_{0})} & \text{if } \mathcal{O} \text{ is repelling. }	
				 \end{cases}
			\end{equation*}
			 
			 Let $L_+$ be the union of $L_{\mathcal{O}}$ over all positive end orbits and $L_-$ be the union of $L_{\mathcal {O}}$ over all negative end orbits. By our construction, $L_+, L_- \subset C$. Observe that the compact regions $\{f^k (L_\pm)\}$ cover $U_\pm$. Thus, \[S \subset C \cup \bigcup_{k\in\Z} f^k (L_+ \sqcup L_-) \subset \bigcup_{k\in\Z} f^k(C).\] It follows from \Cref{coreFiber} that $C$ meets each flow line of $\eL$. 
			 
		\end{proof}
		
		\begin{proof}[Proof of (2)]
			We compute $\chi(C_0)$. Let $Z$ denote the surface $C_0 - \int(\Cmin)$ and let $Z_+$ (resp. $Z_-$) represent the components of $Z$ contained in $U_+$ (resp. $U_-$). Recall that $\chi(C_0)= \chi(\Cmin)+ \chi(Z_+) + \chi(Z_-)$.  We compute $\chi(Z_+)$ working in one positive end orbit $\mathcal{O}$ at a time and note that the computation for $\chi(Z_-)$ is analogous. We let $Z_\mathcal{O}$ denote the components of $Z_+$ contained in $\mathcal{U}_\mathcal{O}$ and $Z_i= Z_\mathcal{O} \,\cap\, \mathcal{U}_{i}$. The surface $Z_i$ is bounded by $\partial\Cmin \,\cap\, \mathcal{U}_i$ and $f^{p-i}(\gamma_{i})$. 
			
			It follows by our construction of $C_0$ that $f^{i}(Z_i)$ embeds into $L_+$ up to isotopy. Further, since $L_+$ is a fundamental domain for the action of $f$ on $\mathcal{U}_+$, $L_+$ embeds into $\partial_+ M$. Finally, observe that the sum $\sum p_{\mathcal{O}}$ taken over all end orbits is bounded above by $\xi(f)=\xi(\partial M)$, since the complexity measures the cardinality of a maximal multicurve of $\partial M$.
			
			Thus, for each end orbit $\mathcal{O}$, \[\chi(Z_\mathcal{O}) \leq \sum_{i=0}^{\pcalO-1} \chi(Z_i)\leq \pcalO\cdot\chi(L_\pm) \leq \pcalO \cdot\chi(\partial_\pm M)= \frac{1}{2}\pcalO\cdot\chi(\partial M) = \frac{1}{3}\pcalO\cdot\xi(f)\]
			
			and
			\[\chi(Z_+ \sqcup Z_-) = \sum_{\mathcal{O}} \chi(Z_\mathcal{O}) \leq \sum_\mathcal{O} \frac{1}{3} \pcalO \cdot \xi(f)\leq \frac{1}{3}\xi(f)^2 \leq \xi(f)^2 \]
			so that  \[\chi(C) +1 \leq \xi(f)^2 + \chi(f) +1 = \capgen(f).\] 
		\end{proof}
		
		\begin{proof}[Proof of (3)-(5)]
			Parts (3), (4) and (5) are by design. Since the junctures of $C_0$ satisfy the juncture distinction property, the core is carried to the properly embedded surface $C=\bar {C_0}$ as in \Cref{coresToGuts}. 
			The gluing homeomorphism $h$ agrees with the hyperelliptic involution from \Cref{rem:h-multiloops} on an annular neighborhood of $\Gamma=\partial C$ in $\partial M$, thus satisfies property (3). Finally, recall that the gluing homeomorphism for {\bf \Cref{case-i}} is chosen in the strongly irreducible setting where $\mathcal{A}= \emptyset$, while that of {\bf \Cref{case-ii}} was engineered to satisfy $h(\mathcal{A}_{ \Sigma})\,\cap\, \mathcal{A}_{ \Sigma} \neq \emptyset$ for each component. In both cases, $N(M,h)$ is atoroidal by \Cref{hatoroidal}. This completes the proof of the proposition.
		\end{proof}

\section{Isometric immersions and proofs of \Cref{thma} and \Cref{thmb}}\label{geoM}
	Before proving Theorems \ref{thma} and \ref{thmb}, we record some general facts about the compactfied mapping torus $M$ and its $h$-double. These results hold more generally for any compact irreducible atoroidal hyperbolic $3$-manifold with boundary of higher genus, but we only state them for our compactified mapping tori.
	
	\begin{lemma}\label{some-immersion}
		Let $h:\partial M \to \partial M$ be a homeomorphism such that $N(M,h)$ is hyperbolic. Then there exists a hyperbolic metric $s\in \AH(M_f)$ and an isometric immersion $\iota: (M_f, s) \to N$ such that for each simple closed curve $\alpha\subset M$, $\ell_{s}(\alpha)= \ell_N(\iota(\alpha))$. 
	\end{lemma}
	\begin{proof}
		Consider the inclusion $i: M\hookrightarrow N$, and let $\pi: N_{\pi_1(M)}\to N$ be the cover of $N$ corresponding to $i_*(\pi_1(M)) \leq \pi_1(N)\leq \Isom^+(\H^3)$. Naturally, $N_{\pi_1(M)}$ is homeomorphic to $M_f$. Let $s$ be the metric corresponding to the pullback of the metric on $N$ under $\pi: M_f \to N$. This gives an isometric immersion $\pi: (M_f,s) \to N$. For each curve $\alpha \subset N$ and each component $\tilde\alpha_0 \subset \tilde\alpha = \pi^{-1}(\alpha)$, the restriction $\pi|_{\tilde\alpha_0}: \tilde\alpha_0 \to \alpha$ is degree-one. Thus, $\ell_s(\tilde\alpha_0) = \ell_N(\alpha)$.	
	\end{proof}
	
	\begin{lemma}\label{tg-immersion}
			Suppose that $f$ is strongly irreducible and let $(M, s^*)$ be its convex-hyperbolic compactified mapping torus with totally geodesic boundary. Let $h: (\partial M, s^*)\to (\partial M, s^*)$ be an isometry and let $N=N(M,h)$ be the $h$-double of $M$ equipped with the hyperbolic metric induced by the gluing. 
			
			Then $(M, s^*) \hookrightarrow N$ isometrically embeds and there is an isometric immersion $\iota: (M_f, s^*) \to N$ so that for each simple closed curve $\alpha\subset M$, $\ell_{s^*}(\alpha)= \ell_N(\iota(\alpha))$. 
	\end{lemma}
	\begin{proof}
		The first claim is immediate by our assumptions. As above, we consider the cover  $\pi: (M_f, s) \to N$ corresponding to $\pi_1(M)\leq \Isom(\H^3)$, where $s$ is the pullback metric from the hyperbolic structure on $N$ induced by the gluing. For concreteness, we check that $(M_f, s)=(M_f, s^*)$. Since $(M, s^*) \subset N$ is a convex totally geodesic submanifold of $N$, its lift $\tilde M\subset \H^3$ is a $\pi_1(M)$-invariant convex subset of $\H^3$ whose boundary consists of totally geodesic embeddings of $\H^2 \hookrightarrow \H^3$. The convex core of $(M_f, s)$ is $\widetilde{M}/\pi_1(M)$, a hyperbolic manifold with totally geodesic boundary. By the uniqueness of such a structure, $s=s^*$. The second claim follows as in the proof of \Cref{some-immersion} above.
	\end{proof}
	
	\subsection{Proofs of \Cref{thma} and \Cref{thmb}}
		We are ready to prove the main theorems.
		\begin{proof}[Proof of \Cref{thma}]
			Fix $D, \varepsilon> 0$ and let $f$ be an atoroidal end-periodic homeomorphism with $\chi(f), \xi(f)\leq D$. Let $C_0$ and $h$ be the choice of core and gluing homeomorphism from \Cref{whichCore}. By \Cref{goodCore}, $C_0$ and $h$ satisfy the hypotheses of \Cref{spf}. Further, by \Cref{some-immersion}, there is a hyperbolic structure $s\in \AH(M_f)$ for which $(M_f, s)$ isometrically immerses into the closed fibered hyperbolic manifold $N=N(M,h)$. It follows from \Cref{spf}\ref{spf:Ngenus} and \Cref{ourspf} that $N$ is fibered by a closed surface $F$ isotopic to $\hat F(C,h)$ whose genus is bounded above by $\capgen(f)\leq D^2+D+1$ (see \Cref{lmtfiber} and \Cref{goodCore}).
			
			Let $K'=K'(D,\varepsilon)$ be the constant from \Cref{minsky} coming from a surface of genus $D^2 + D + 1$ and let $K=\textrm{ max}\{K', 4\} + 2$. Let $Y$ be a subsurface of $S\subset M$ satisfying $d_Y(\Lambda^+,\Lambda^-)\geq K$. Then by \Cref{spf}\ref{spf:lamicom},  $Y$ is isotopic to $W \subset F$ so that $d_W(\lambda^+, \lambda^-) \geq K'$.
			We apply Minsky's \Cref{minsky} so that $\ell_N(\partial W) \leq \varepsilon$. 
			It follows by \Cref{some-immersion} that $\ell_s(\partial Y) = \ell_s(\partial W) =\ell_N(\partial W) \leq \varepsilon$, completing the proof of \Cref{thma}. 
		\end{proof}
		
		\begin{proof}[Proof of \Cref{thmb}]
			This theorem is a special case of \Cref{thma}, we repeat a nearly identical argument to that above. Fix $D, \varepsilon > 0$ and let $f$ be a strongly irreducible end-periodic homeomorphism whose $\partial M$ consists of genus-$2$ surfaces. Note that the restriction on the topology of $\partial M$ ensures that the mapping class is pure. Let $C$ and $h$ be the properly embedded surface and gluing homeomorphism chosen from {\bf \Cref{case-i}}. Recall that $h$ was chosen as a hyperelliptic involution about $\Gamma=\partial C$ and that $C_0$ and $h$ satisfy the hypotheses of \Cref{spf}.
			
			Since each component of $\partial M$ is a genus-$2$ surface, the mapping class $h$ preserves each isotopy class of curve in $\partial M$. Thus, by the $9g-9$ Theorem, $h|_{\Sigma}$ acts trivially on the Teichm\"uller space of $\Sigma$ for each $\Sigma\subset \partial M$ (see \cite[Section 12.1]{FM12}). Let $\hat h$ be the isometry of $(\partial M, s^*)$ which is isotopic to $h$. Now consider $N^*= N^*((M, s_*), \hat h)$, the $\hat h$ double of $(M, s_*)$. Of course, since $h$ and $\hat h$ are isotopic, $N$ and $N^*$ are homeomorphic, thus, isometric by Mostow Rigidity \cite[Theorem 12.2]{Mostow1968}. 
			By \Cref{tg-immersion}, $(M_f, s^*)$ isometrically immerses into $N^*=N$ so that $\ell_{s^*}(\alpha) = \ell_N(\alpha)$ for all simple closed curves $\alpha$ of $M_f$. 
			
			Let $K'$ be the constant from \Cref{minsky} coming from $\varepsilon$ and a surface of genus $D^2+D+1$. Exactly as the proof above, by applying \Cref{spf} and \Cref{minsky} for  $K=\textrm{ max}\{K', 4\} + 2$, it follows that  if $d_Y(\Lambda^+, \Lambda^-)\geq K$, then $\ell_{s^*}(\partial Y)  = \ell_{N^*}(\partial Y)\leq \varepsilon$. This completes the proof of \Cref{thmb}.
		\qedhere
		\end{proof}
\section{Examples}\label{examples}
We detail several constructions of end-periodic homeomorphisms before providing a few applications of our main theorems. We kindly note to the reader that much of the notation has been reset. 

\subsection{Handle shift mapping classes}
	First we explicitly describe a somewhat natural occurrence of an end-periodic translation on a ladder surface. Let $\Sigma=\Sigma_g$ be a closed connected surface of genus $g$. Let $R = \Sigma -\beta$, where $\beta$ is a non-separating loop on $\Sigma$ (see \Cref{letrhobe}). Consider the following homomorphism
	$$\tau_\beta: \pi_1(\Sigma) \to H_1(\Sigma) \xrightarrow{\iota_\beta} \Z$$
	which is the composition of the abelianization map on $\pi_1(\Sigma)$ and the algebraic intersection form with respect to $\beta$ defined by $\iota_\beta([\alpha]) = \langle [\alpha], [\beta] \rangle$. We observe that $H = \ker{\tau_\beta} = \ll \pi_1(R)\gg$, the normal closure of $\pi_1(R)$. The cover $\tilde \Sigma_H \to \Sigma$ corresponding to $H$ has a cyclic deck transformation group generated by $T: \tilde \Sigma_H \to \tilde \Sigma_H$. 
	
	\begin{example}[A translation map on the ladder]\label{hsladder}
		Given $\Sigma$, $\beta$ and $H$ as above:
		\begin{enumerate}[label=(\roman*)]
			\item $\tilde \Sigma_H$ is homeomorphic to the ladder surface $L$.
			\item $T$ is an end-periodic translation map. 
		\end{enumerate}
	\end{example}
	
	\begin{proof}
		First note that $H= \pi_1(\Sigma_H)$ is an infinite-index subgroup of $\pi_1(\Sigma)$, so $\tilde \Sigma_H$ is a boundaryless surface of infinite-type. Choose a component of $p^{-1}(R)$ and let $\tilde R$ denote its closure. The surface $\tilde R$ is a homeomorphic lift of $\Sigma \bbslash \beta$ and is a fundamental domain of the infinite-order deck transformation action.  Since $\langle T \rangle$ acts freely and cocompactly on the surface, the \v Svarc--Milnor lemma (see \cite[Prop. 8.19]{BH99}) implies that $\tilde \Sigma_H$ is two-ended as it is quasi-isometric to $\Z$. It follows that $\tilde \Sigma_H$ is homeomorphic to the ladder surface, as it is the unique boundaryless two-ended surface of infinite type. 
			
		 We have the neighborhood $U_+$ defined as the interior of $\bigcup_{k\geq 0} T^k (\tilde R)$ so that $T(U_+) \subset U_+$. We similarly define $U_-$  as the interior of $\bigcup_{k \leq 0} T^k (\tilde R)$ with $T^{-1}(U_-) \subset U_-$. Note that iterates of each set form neighborhood bases for the respective ends and $\bigcup_{n\in\Z} T^n(U_+)=\bigcup_{n\in\Z} T^n(U_-) = L$. Thus, $T$ is an end-periodic translation map. 
	\end{proof}

\begin{figure}[h]
	\centering
	\begin{tikzpicture}[>=latex]
		\centering
		\node at (0,0) {\includegraphics[width=.5\textwidth]{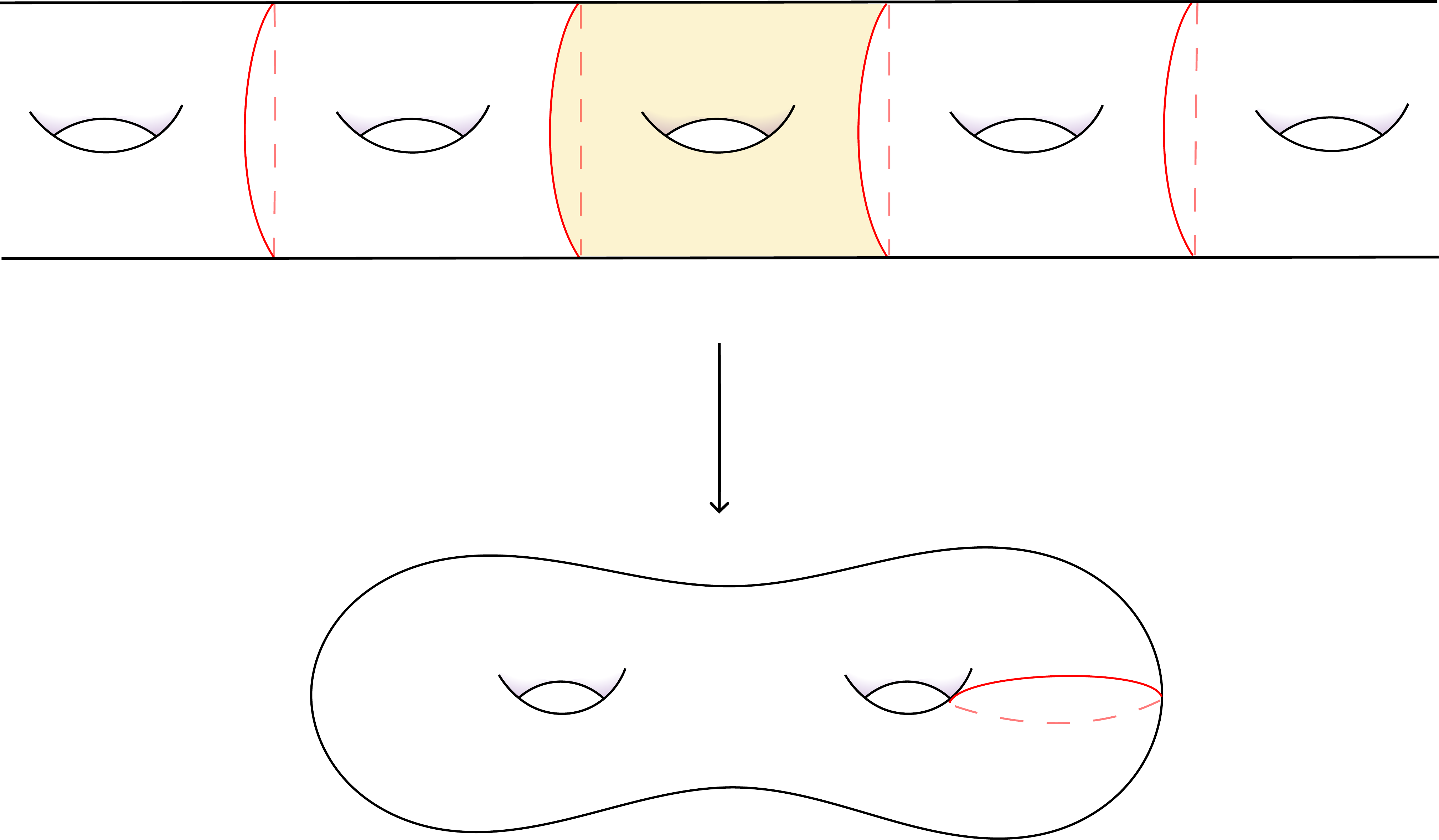}};
		\node at (-.7,.4) {\scalebox{0.6}{\color{red}{$\tilde \beta_0$}}};
		\node at (-2.1,.4) {\scalebox{0.6}{\color{red}{$\tilde \beta_{-1}$}}};
		\node at (.7,.4) {\scalebox{0.6}{\color{red}{$\tilde \beta_1$}}};
		\node at (2.1,.4) {\scalebox{0.6}{\color{red}{$\tilde \beta_2$}}};
		\node at (2.4,-1.35) {\scalebox{0.6}{\color{red}{$\beta$}}};
	\end{tikzpicture}
		\caption{The curve $\beta$ lifts to the $T$-invariant, infinite-component multicurve $\displaystyle \tilde \beta= \bigcup_{n\in\Z} \tilde\beta_n$ colored in red. Let $C_i$ denote the core bounded by $\tilde\beta_i$ and $\tilde\beta_{i+1}$. The core $C_0$ is highlighted in yellow. \label{letrhobe}}
\end{figure}
	We note that any such map above is isotopic to a handle shift homeomorphism as defined by Patel--Vlamis \cite{PV18}.  Let $\bar L$ be the surface whose boundary is homeomorphic to $\R \,\sqcup\, \R$ and whose interior is homeomorphic to the Loch Ness Monster surface (the unique boundaryless infinite-type surface with one end). Informally, a handle shift homeomorphism is an embedding of a translation map of $\bar L$ into $S$, supported on the interior of the image of the embedding. We refer the reader to \cite[Section 6]{PV18} for their precise definition of a handle shift homeomorphism. It is worth noting that handle shift homeomorphisms themselves are not necessarily end-periodic, nor isotopic to an end-periodic homeomorphism.

\subsection{Constructions of end-periodic homeomorphisms}
	The following proposition of Field--Kim--Leininger--Loving uses handle shifts and compactly supported mapping classes to produce plentiful examples of strongly irreducible homeomorphisms. The authors specify $\rho$ to be the product of commuting handle shift homeomorphisms with disjoint support as in \cite[Construction 2.10]{FKLL23}. We rephrase their result below with different conditions as their arguments work for a more general class of end-periodic homeomorphisms.
	\begin{proposition}[Field--Kim--Leininger--Loving {\cite[Prop. 6.3]{FKLL23}}]\label{fkllexamples}
		Given an end-periodic homeomorphism $\rho$ , a connected core  $C$ of $\rho$ and $g\in \Homeo(C, \partial C)$, if:
		\begin{enumerate}[label=(\arabic*)]
			\item $\partial C$ separates each end, i.e. the components of $S-C$ are in bijection with $\Ends(S)$, 
			\item $\rho(\partial_-C)$ and $\rho^{-1}(\partial_+C)$ are disjoint multicurves and each of their components is essential in $C$, and 
			\item $ d_C(\rho(\partial_-C), g\rho(\partial_- C))> 9$,
		\end{enumerate}
		then $f= \rho g$ is a strongly irreducible end-periodic homeomorphism. 
	\end{proposition}
	 Given a choice of end-periodic map $\rho$, the above conditions are satisfied with ease by extending the core $C$ and taking sufficiently large powers of any pseudo-Anosov $g$ as needed.\\
	 Consider a family of atoroidal end-periodic homeomorphisms of the form $f_n=\rho g^n$, constructed as above for $n\neq 0$ where $g\in \Homeo(C, \partial C)$ is a representative of a pseudo-Anosov mapping class. It follows that each $f_n$ has the same end complexity, and since $C$ remains a core of each, its core characteristic is bounded above by $\chi(C)$. Let $\Lam^+_n$, $\Lam^-_n$ denote the Handel--Miller laminations of $f_n$. The following proposition shows that $d_C(\Lambda^+_n, \Lambda^-_n) \to \infty$ as $n\to \infty$.
	\begin{proposition}\label{largesubproj}
		Let $\rho$ be an end-periodic homeomorphism with juncture orbits $P_+$, $P_-$. Let $Y = \bigsqcup Y_i$ be the union of finitely many connected essential subsurfaces.	
		Let $g \in \Homeo(Y, \partial Y)$ be so that $g|_{Y_i}$ represents a fully supported mapping class on each connected component $Y_i\subset Y$.  Consider the end-periodic homeomorphism $f=\rho g$ with its Handel--Miller laminations $\Lam^+, \Lam^-$.
		Suppose that:	
		\begin{enumerate}[label=(\arabic*)]
			\item  $f$ is not a translation; and
			\item each $Y_i$ meets $\Lambda^+$, $\Lambda^-$, $P^+$ and $P^-$.
		\end{enumerate}
		Then, 
		$$\tau_{Y_i}(g)- d_{Y_i}(P_+, P_-) -2 \leq d_{Y_i}(\Lambda^+,\Lambda^-). $$
	\end{proposition}
	
	Recall that $\tau_{Y_i}(g)$ is the stable translation length of $g$ in $\mathcal{AC}(Y_i)$, and that any fully supported mapping class has positive translation length with $\tau(g^n) = |n\tau(g)|$.  In the special case when $\rho$ is a translation, one may take $P_+ = P_-$ so that $d_{Y_i}(P_+, P_-)=0$. 
	
	\begin{proof}
		We work in one component $Y_i$ at a time. Let  $\alpha, \eta$ be a pair of positive and negative junctures in $P_+$, $P_-$ respectively so that  $\rho^{-1}(\alpha)$ and $\rho(\eta)$ meet $Y_i$, but for $k\geq 0$, $\rho^{k}(\alpha)$ and $\rho^{-k}(\eta)$ are disjoint from $Y_i$ (this exists by the compactness of $Y$, see the related \Cref{lemma:junctureminl}). It follows that $\alpha$ and $\eta$ define juncture orbits $J_+ = \bigcup_{k\in\Z} f^k(\alpha)$ and $J_- = \bigcup_{k\in\Z} f^{k}(\eta)$ of $f$. Recall from \Cref{lemma:comparelam} that $d_{Y_i}(J_+,J_-) - 2 \leq d_{Y_i}(\Lambda^+, \Lambda^-)$.
		\begin{align*}
			\tau_{Y_i}(g) =\tau_{Y_i}(g^{-1}) &\leq d_{Y_i}(g^{-1}p^{-1}(\alpha), p^{-1}(\alpha)) \\ 
			&\leq d_{Y_i}(g^{-1}p^{-1}(\alpha), \rho(\eta)) + d_{Y_i}(\rho(\eta), \rho^{-1}(\alpha)).\\
		\end{align*}
		We observe from the properties above that $\pi_{Y_i}(\rho(\eta))=\pi_{Y_i}(\rho g(\eta))= \pi_{Y_i}(f(\eta))$.
		So, continuing the calculation from the previous line, we have:
		\begin{align*}
		&= d_{Y_i}(f^{-1}(\alpha), f(\eta)) + d_{Y_i}(\rho(\eta), \rho^{-1}\alpha)\\
		&\leq  d_{Y_i}(J_+, J_-) + d_{Y_i}(P_+ , P_-)\\
		&\leq  d_{Y_i}(\Lambda^+,\Lambda^-) +2 + d_{Y_i}(P_+, P_-).
		\end{align*}
		This concludes the proof.
	\end{proof}

\subsection{Setup}
	For the next few examples, we fix the following setup:
	Let $L$ be the ladder surface and $\rho=T$ be the translation map corresponding to $\Sigma_2$ and $\beta$ as constructed in \Cref{hsladder} and illustrated in \Cref{letrhobe}. We label each component of $\tilde \beta$ by the integers as in the figure above and let $C_i$ denote the core bounded by the separating curves $\tilde\beta_i$ and $\tilde\beta_{i+1}$.

\subsection{Short non-peripheral curves}
	 The Handel--Miller laminations and juncture orbits give an obstruction to a curve being peripheral in the manifold. Recall that a loop $\gamma \subset S$ is positively (resp. negatively) escaping with respect to $f$ if and only if it is contained in $\mathcal{U_+}$ (resp. $\mathcal{U}_-$), and that any escaping loop of $f$ is peripheral in $M_f$.  The following remark follows from \Cref{escaping}. 
	
	\begin{remark}\label{noEscape}
		Let $\alpha$ be a curve in $S$. If there is a subsurface $W$ for which $d_W(\Lambda^+, \alpha)\geq 2$, then $\alpha$ is not negatively escaping. Similarly, $d_W(\Lambda^-, \alpha)\geq 2$ implies $\alpha$ is not positively escaping.
	 \end{remark}
	 
	\begin{example}[Short non-peripheral curves]\label{pinchint}
		Let $L$ be the ladder surface and suppose $Y$ is the disjoint union of essential subsurfaces $Y_1, \ldots, Y_n$. Then for every $\varepsilon>0$, there exists $f=f_\varepsilon$, a strongly irreducible end-periodic homeomorphism, and associated mapping torus $M_f$ so that:
		\begin{enumerate}[label=(\arabic*)]
			\item no component of $\partial Y$ is peripheral in $\Mf$; and 
			\item the corresponding structure on $(M_f, s^*)$ whose convex core has totally geodesic boundary satisfies 
			$$ \ell_{s^*}(\partial Y) \leq \varepsilon .$$ 
		\end{enumerate}
	\end{example}
	
	\begin{proof}
		 We will first construct an auxiliary end-periodic map $f_0$ for which $\Gamma=\partial Y$ is non-escaping.
		Let $\rho$ be the translation with fundamental domains $C_i$ depicted in \Cref{letrhobe}. We direct the reader to \Cref{projcalculations} to remind them of a few facts which we will use for subsurface projection calculations throughout this argument and the rest of the section. 
		
		Let $C'$ be a union of consecutive $C_i\cup \ldots\cup C_{i+k}$ which contains $Y$ as an essential subsurface. We extend $C'$ to $C= C_{i-1} \,\cup\, \ldots \,\cup\, C_{{i+k+1}}$ so that $\pi_C(\rho(\Gamma))$, $\pi_C(\rho^{-1}(\Gamma)) \neq \emptyset$ and the multicurves $\rho(\partial_-C)$, $\rho^{-1}(\partial_+C)$, and $\Gamma$ are mutually disjoint and contained in $C$. 
		 By considering sufficiently large powers of some pseudo-Anosov mapping class on $C$, we choose $h \in \Homeo(C,\partial C)$ such that: 
			\begin{enumerate}[label=(\roman*)]
				\item $d_{C}(h\rho(\partial_- C), \rho(\partial_- C)) \geq 9$. 
			\end{enumerate}	
		Note that this implies:
			\begin{enumerate}[label=(\roman*)]
				\setcounter{enumi}{1}
				\item $d_C(h \rho(\partial_- C), \rho^{-1}(\Gamma)) \geq 7$ and
				\item $d_C(h^{-1}(\Gamma), \Gamma)\geq 7$.	
			\end{enumerate} 
		From the conditions above and \Cref{fkllexamples}, it follows that $f_0=\rho h$ is strongly irreducible. Observe that $C$ must meet both $\Lam_0^+$ and $\Lam_0^-$, the Handel--Miller laminations of $f_0$, by \Cref{mustMeetLam}; otherwise some component of $\partial C$ is reducing. With a direct application of \Cref{noEscape}, we see that a component of $\Gamma$ is neither negatively nor positively escaping under the strongly irreducible map $f_0$, since its Handel--Miller laminations satisfy: 
			\begin{align}\label{GammaNoEscape}
				\begin{split}
					d_{\rho C}(\Lambda_0^+, \Gamma) &\geq d_{\rho C}(f_0^{2}(\partial_- C), \Gamma)-1\\	 	 
					&= d_{\rho C}(\rho h \rho (\partial_- C), \Gamma) -1 \\
					&\geq 6,\\
				 \end{split}
			\end{align}
			
			\begin{align}\label{GammaNoEscape1}
				\begin{split}
					d_{C}(\Lambda_0^-, \Gamma) &\geq d_{C}(f_0^{-1} (\partial_+ C), \Gamma) -1 \\
					&\geq d_C(\Gamma, h^{-1}(\Gamma)) - d_C( h^{-1}(\Gamma), h^{-1}\rho^{-1}(\partial_+ C)) -1\\ 
					&\geq 5.
				\end{split}
			\end{align}
		For each component $Y_i$ of $Y$, let $g_i \in \Homeo(Y_i, \partial Y_i)$ be a representative of a fully supported mapping class. We let $g=\prod g_i$ be the composition over all components. Consider $f_n = \rho h g^n=f_0g^n$ with its associated Handel--Miller laminations $\Lambda^+_n$, $\Lambda^-_n$. Recall that each $f_n$ has bounded capacity and that $g$ preserves $\partial_\pm C$. For sufficiently large n, it follows that $f_n= \rho h g^n$ satisfies the conditions of \Cref{fkllexamples}, hence is a strongly irreducible homeomorphism. Our goal is to apply \Cref{largesubproj} to $f_0$ and $g^n$ to show that $d_{Y_i}(\Lambda^+_n, \Lambda^-_n)$ is arbitrarily large for each $Y_i \subset Y$.
		
		Let $\eta= \partial_-C$ and $\alpha =\partial_+ C$. Note that $\alpha$, $\eta$ are positive and negative junctures of $f_n$ for all $n$. The equations \cref{GammaNoEscape} and \cref{GammaNoEscape1} with the conditions above on $h$ ensure that every component of $\Gamma$ meets both $f_0^{-1}(\alpha)$ and $f_0^2(\eta)$ (recall that $\Gamma \subset \rho(C)$). Thus the $f_0$-juncture orbits for $\alpha$ and $\eta$, $$P^+ = \bigcup_{k\in\Z} f_0^k(\alpha), \quad P^-= \bigcup_{k\in\Z} f_0^{k}(\eta),$$ both meet $Y_i$ for all $i$ (see \Cref{projcalculations}). In the following calculations we use the same logic to confirm that $\Lam_n^+$ and $\Lam_n^-$  meet $Y_i$, this verifies that $\Gamma$ is neither positively nor negatively escaping under any $f_n$.
		\begin{align*}
			\begin{split}
				d_{C}(\Lambda_n^-, \Gamma) &\geq d_{C}(f^{-1}_n(\alpha), \Gamma) -1 \geq d_{C}(g^{-n}h^{-1}\rho^{-1} (\alpha), \Gamma) -1 \\
				&\geq d_{C}(f_0^{-1} (\alpha), \Gamma)-1 \geq 5.
			\end{split}
		\end{align*}
		
		For this calculation, let $W= \rho h(C)$. Recall that $g$ and $h$ are supported on $Y$ and $C$ respectively, that $\rho(\eta)$ and $\rho^{-1}(\alpha)$ are both disjoint from $\Gamma$ and that $d_C(\Gamma, \rho^{\pm}(\Gamma))\leq 2$.
		\begin{align*}
			 d_{W}(\Lam_n^+, \Gamma) &\geq d_{W}(f_n^{2}(\eta), \Gamma) -1\\ 
			 &=d_W(\rho h g^n \rho h g^n (\eta), \Gamma) -1 = d_W(\rho h g^n \rho (\eta), \Gamma) -1 \\
			 &= d_C(g^n \rho (\eta), h^{-1}\rho^{-1}(\Gamma))-1  \\
			&\geq -1 -d_C(g^n \rho (\eta), g^n (\Gamma)) + d_C(g^n(\Gamma), h^{-1}(\Gamma))- d_C(h^{-1}(\Gamma), h^{-1}\rho^{-1}(\Gamma)) \\
			&= d_C(\Gamma, h^{-1}(\Gamma)) -4 \geq 3. \\
		\end{align*}
		So, $f_0$ is an end-periodic homeomorphism whose juncture orbits $P^+$, $P^-$ meet $Y_i$ for each component $Y_i \subset Y$. For each $n$, the Handel--Miller laminations $\Lam_n^+$, $\Lam_n^-$ both meet $Y_i$. By \Cref{largesubproj}, we have:
		\[\tau_{Y_i}(g^n) - d_{Y_i}(P^+, P^-) -2 \leq d_{Y_i}(\Lambda^+_n,\Lambda^-_n).\]
		
		Now, since $g|_{Y_i}\in \Map(Y_i)$ is fully supported, it acts loxodromically on $\mathcal{AC}(Y_i)$. Let $c$ be the constant from \Cref{mintrans} such that $c|n| < \tau_{Y_i}(g^n)$ for all $i$. Let $B = \max \{d_{Y_i}(P^+, P^-) +2\}$. 
		
		Finally, let $\varepsilon >0$, $D = \max\{\xi(f_0), \chi(f_0)\}$ and let $K=K(D, \varepsilon)$ be the constant from \Cref{thmb}. Above we verified that $f_n = \rho g^n h$ is a strongly irreducible homeomorphism for all $n$ with $\partial \overline{M_{f_n}} \cong \Sigma^+_2 \,\sqcup\,\Sigma^-_2$. Choose $N$ so that $K \leq c|N|-B-2$. Let $M_N =(\M_{f_N}, s^*)$ be the corresponding convex-hyperbolic mapping torus with totally geodesic boundary. Since $d_{Y_i}(\Lambda_N^+, \Lambda_N^-) \geq K$, by \Cref{thmb} we have $\ell_{s^*}(\partial Y_i) \leq \varepsilon$ as required.  
	\end{proof}

\subsection{Thin totally geodesic structures}
	Our results in this section mirror Kent's main result in \cite{K05} which produces arbitrarily short curves of totally geodesic boundaries of knot complements. Let $\Sigma^+=\Sigma^+_g$ and $\Sigma^-=\Sigma^-_g$ each be a closed surface of genus $g$.
	Suppose that $M$ is the acylindrical compactified mapping torus of a strongly irreducible end-periodic map whose boundary $\partial_\pm M = \Sigma^\pm_g$. We let $(\Sigma^\pm, s^*)$ denote the totally geodesic metric on the boundary surface. We let $\ell_{s^*}(\Sigma)$ denote the length of shortest closed (simple) geodesic of $(\Sigma, s^*)$, i.e. the length of its systole.

\subsubsection{Setup} 
	We use $C = C_{-1} \,\cup\, C_0 \,\cup C_1$, the core of $\rho$ on the ladder surface L as in \Cref{letrhobe}. Let $P_+ = P_- = \bigcup_{i\in\Z} \rho^i (\partial_-C_0)$. Let $\phi \in \Homeo(C,\partial C)$ be a representative of a pseudo-Anosov mapping class of $C$. 
	\begin{remark}[Some bookkeeping]\label{bookkeeping}
		Let $\varepsilon>0$, $D = \max\{|\chi(C)|, 2\xi(\Sigma_{2})\}$ and $K = K(D, \varepsilon)$ be the corresponding constant from \Cref{thmb}.  
		There is a power $N$ so that for $n\geq N$, $\rho \phi^n$ is strongly irreducible with  $\tau(\phi^n) \geq K(D,\varepsilon)+3$. Considering all powers of such maps, let \[f(n,g) := {(\rho \phi^n)}^{g-1}\] for $n,g \in\Z$ and let $\Lambda^+(n,g),\, \Lambda^-(n,g),$ denote the corresponding positive and negative Handel--Miller laminations of $f(n,g)$. 
		
		Observe that for all $n\geq 1,\, g\geq 2$, we have $\xi(f(n,g))= 2\xi(\Sigma_{g})$ and $\chi(f(n,g)) \leq |\chi(C)|=6$.
	\end{remark}
	\begin{example}[Thin totally geodesic structures]\label{pinch}
		For every $g\geq 2$, $\varepsilon >0$, there exists $N>0$ so that for all $n\geq N$, the composition $f(n,g)= (\rho \phi^n)^{g-1}$ satisfies the following.
		
		\begin{enumerate}[label=(\roman*)]
			\item  The corresponding compactified mapping torus $M(n,g)$ is an acylindrical convex-hyperbolic manifold.
			\item  The boundary surfaces $\partial M(n,g) \cong \Sigma^+_g \sqcup \Sigma^-_g$, where each component is a closed genus-$g$ surface. 
			\item The components of $\partial C$ satisfy $\ell_{s(n,g)^*}(\partial C) \leq \varepsilon$, where $s(n,g)^*$ denotes the convex-hyperbolic structure on $M({n,g})$ with totally geodesic boundary.
		\end{enumerate}
	\end{example}
	
	In other words, this example produces hyperbolic $3$-manifolds which admit totally geodesic boundary structures of any genus $\geq 2$ and arbitrarily short systole. To make the notation easier on the eyes, we will drop the $^*$ decoration so that any $s$ denotes the totally geodesic structure of the associated $3$-manifold. 
	
	\begin{proof}
		We first prove the statement when $\partial M$ is the disjoint union of closed genus-$2$ surfaces.	Fix $\varepsilon >0$, let $D$, $K$ and $f_2= f(N,2)= \rho \phi^N$ be as above in \Cref{bookkeeping}. With \Cref{largesubproj}, we have $d_C(\Lambda^+, \Lambda^-)\geq K$. A direct application of \Cref{thmb} has that $\ell_{s}(\partial C) \leq \varepsilon$ in the corresponding totally geodesic structure $s$ for $f_2$. This concludes the genus-$2$ case. 
	
		To generalize to the genus-$g$ setting, consider the $(g-1)$-cover of $M_2=M_{f_2}$ corresponding to the compactified mapping torus of $f_g=(f({N,2}))^{g-1}$. Let $M_g = M_{f_g}$ and observe that $\partial M_g \cong \Sigma^+_g \sqcup \Sigma^-_g$. Indeed, the maps $f_g$ and $\rho^{g-1}$ agree on the ladders $U_\pm$ defined by $S-C= U_+ \cup U_-$. Thus, the quotient of either escaping set by the action of $f_g$ is homeomorphic to  $L / \langle \rho^{g-1} \rangle$, a genus $g$ surface.
		
		Geometrically, the covering map ${p}: (M_{f_g}, s_g) \to (M_{2}, s)$ is obtained by the quotient $\H^3 /\pi_1(M_{g}) \to \H^3 / \pi_1(M_2)$, where $\pi_1(M_{g})\leq \pi_1(M_{2})$ is given by taking $g-1$ powers of $f$. Observe that the boundary of the convex core $(\partial C(M_g), s_g) \hookrightarrow (M_g, s_g) $ is totally geodesic, hence $s_g$ is the unique such structure.
		
		Let $\alpha = \partial_+ C$ and $\eta = \partial_- C$. For each component $\tilde\alpha_0$ of $p^{-1}(\alpha)$, the restricted cover $p|_{\tilde\alpha_0}$, is degree one. Thus, $\ell_{s_g}(\tilde\alpha_0) = \ell_{s}(\alpha) \leq \varepsilon$, as shown in the genus-$2$ case. An identical argument shows $\ell_{s_g}(\tilde\eta_0)\leq \varepsilon$ for any component $\tilde \eta_0$ of $\tilde\eta$. Thus, for the strongly irreducible map $(\rho \phi^N)^{g-1}$,  we have that the systoles $\ell_{s_g}(\Sigma_+), \ell_{s_g}(\Sigma_-) \leq \varepsilon$. 
	\end{proof}
	
	We conclude with a proof of \Cref{thmc}. 
	\begin{proof}[Proof of \Cref{thmc}]
		We reintroduce the same characters from the previous example. Fix $\varepsilon>0$, $g\geq 2$ and let $D$, $K$ and $N$ be the constants from \Cref{bookkeeping}. Consider the map $f=\rho \phi^N$, now with its compactified mapping torus $(M,s)$ with totally geodesic boundary realized as the convex core of $M_f$.
		Once again, we consider the compactified mapping torus of $f^{g-1}$ with corresponding convex hyperbolic metric $(\tilde M, \tilde s)$ inherited from the convex core of the associated cover. Let $p: (\tilde M, \tilde s) \to (M, s)$ denote this $(g-1)$-fold cover.
		
		We choose a core and gluing homeomorphism $h$ just as {\bf \Cref{case-i}} from \Cref{whichCore}. In particular, recall that $h=\hat h: (\partial M, s) \to (\partial M, s)$ is an isometry and the choices satisfy the conditions of \Cref{spf}. It follows from \Cref{ourspf} that the $h$-double $N = N(M, h)$ is hyperbolic and fibered by a closed surface $F$. Observe that $(\partial M, s)$ is totally geodesic in both $(M, s)$ and $N$, as it retains its totally geodesic structure from the gluing.
	   	
	   	\begin{claim}
	   	The isometry $h$ lifts to an isometry $\tilde h: (\partial \tilde M, \tilde s) \to (\partial \tilde M, \tilde s)$. 
	   	\end{claim}
			\begin{proof}[Proof of claim]
			For each component of $\tilde \Sigma_0 \subset \tilde \Sigma=p^{-1}(\Sigma)$, the covering map $p|_{\tilde \Sigma_0}$ corresponds (via the Galois correspondence) to the kernel of the following homomorphism:
		    \[\pi_1(\Sigma) \to H_{1}(\Sigma, \Z) \xrightarrow{\iota_\beta} \Z \xrightarrow{\mod{g-1}}  \Z_{g-1}. \]
		    Since $h$ acts by $-1$ on $H_1 (\Sigma)$, it preserves the corresponding kernel subgroup, thus it lifts to a diffeomorphism $\tilde h$ by the lifting criterion. We see that $\tilde h$ is a global isometry since it is a local one. 
		   \end{proof}
	   
   Let $\tilde N$ be the hyperbolic $\tilde h$-double of $\tilde M$. Recall as in \Cref{constructN} that the identity map $\sigma: \tilde M \to \tilde{M}^{\downarrow}$ is orientation reversing. The cover $p: \tilde M \to M$ and gluing map $\tilde h: \partial \tilde M \to \partial \tilde M$ extend to a cover $\pi: \tilde N \to N$  defined by $\pi|_M := p$ and $\pi|_{M^{\downarrow}} := \sigma p \sigma^{-1}$. We see that $(\partial \tilde M, \tilde s)$ is totally geodesic in both $(\tilde M, \tilde s)$ and $\tilde N$.
   The $1$-dimensional foliation $\eL$ of $N$ lifts to a $1$-dimensional foliation $\tilde{\eL}$ of $\tilde N$ of which $\tilde F = \pi^{-1}(F)$ is a cross section. Further, the embeddings $\partial M \hookrightarrow M \hookrightarrow N$ lift so that for either component $\tilde \Sigma_g \subset \partial \tilde M$, we have:
	
		\begin{center}
			\begin{tikzcd}
			 (\tilde \Sigma_g, \tilde s) \arrow[d]\arrow[hookrightarrow]{r}& (\tilde M,\tilde s)  \arrow[hookrightarrow]{r} \arrow[d]{p}  & \tilde N \arrow[d]    \\
			 (\Sigma_2, s) \arrow[hookrightarrow]{r} &	(M ,s)  \arrow[hookrightarrow]{r}          & N.
			\end{tikzcd}
		\end{center}
	Observe that each embedding is isometric and the diagram commutes. \Cref{pinch} shows that  $\ell_{\tilde s}(\partial C) \leq \varepsilon$, thus $\ell_{\tilde s}(\tilde \Sigma_g)\leq \varepsilon$. 
	
	In summary, $\tilde N$ is a closed, fibered hyperbolic manifold with a totally geodesic embedding of (two copies of) $\tilde \Sigma_g$, a hyperbolic genus-$g$ surface with $\ell(\tilde \Sigma_g) \leq \varepsilon$.  The spinning construction of Landry--Minsky--Taylor \cite[Section 3.3]{LMT23} directly shows that $\Sigma$ can be made transverse to the almost pseudo-Anosov flow $\varphi^\flat$ given by the spun pseudo-Anosov package $(N, \mathcal{F}, \varphi^\flat, M)$ (see \Cref{ourspf}). \qedhere
\end{proof}

	\subsection*{Acknowledgements}
		The author thanks their advisor, Sam Taylor, for introducing them to this fascinating story of depth-one foliations and spun pseudo-Anosov maps, for suggesting the main problem and for his guidance and patience during the research and writing process. The author additionally thanks a few anonymous referees whose comments have significantly improved earlier drafts of this paper.
		
		The author gratefully acknowledges support of NSF RTG: Geometry, Group Actions, and Dynamics at
		Wisconsin via DMS-223090 and completed this work during their time as a Strauch Postdoctoral Fellow at the Simons Laufer Mathematical Sciences Institute (the institute formerly known as MSRI). 
	
	\bibliographystyle{amsalpha}
	\bibliography{ref}

\end{document}